\documentclass[oneside, reqno]{amsart}

\newcommand{\hide}[1]{}

\usepackage{amssymb}
\usepackage{graphicx}
\usepackage{amsmath}

\usepackage{enumitem}
\usepackage{bm}

\usepackage{mathtools}
\usepackage{graphics}
\usepackage{xcolor}
\usepackage{multicol}

\usepackage{textpos}
\usepackage{subfigure}
\usepackage{framed} 
\usepackage{tikz}
\usepackage{tikz-cd}
\usepackage{commath}
\usepackage{nicematrix}
\usepackage{geometry}
\usepackage{float}

\usepackage[doi=false,isbn=false,url=false,eprint=false, style=alphabetic, backend=bibtex,maxbibnames=99]{biblatex}
\usetikzlibrary{shapes, arrows.meta, positioning}

\usepackage[colorlinks=true,urlcolor=blue]{hyperref}

\DeclareMathAlphabet{\mathbbold}{U}{bbold}{m}{n}

\newcommand{\ITM}{\mathrm{ITM}}
\newcommand{\IET}{\mathrm{IET}}

\newcommand{\eps}{\epsilon}

\newcommand{\R}{\mathbb{R}}

\newcommand{\betas}{\beta_*}
\newcommand{\betass}{\beta_{**}}

\renewcommand{\ge}{\geqslant}
\renewcommand{\le}{\leqslant}

\theoremstyle{plain}
\newtheorem{theorem}{Theorem}[section]
\newtheorem{proposition}[theorem]{Proposition}

\newtheorem*{conjecture*}{Conjecture}
\newtheorem*{perturbation-lemma}{First Return Map Perturbation Lemma}
\newtheorem*{lin-dep-ret-vec}{Transversality of Itinerary Subspaces}
\newtheorem*{mconjecture}{Boshernitzan--Kornfeld Conjecture}

\newtheorem*{theorem*}{Theorem}
\newtheorem{lemma}[theorem]{Lemma}
\newtheorem{corollary}[theorem]{Corollary}

\theoremstyle{remark}
\newtheorem{question}[theorem]{Question}

\theoremstyle{definition}
\newtheorem{definition}[theorem]{Definition}

\newtheorem{remark}[theorem]{Remark} 

\numberwithin{figure}{section}

\newcounter{lstv}
\newenvironment{lstv}{%
\refstepcounter{lstv}%
\begin{center}
\begin{minipage}{.9\textwidth}}{%
\end{minipage}%
\makebox[.1\textwidth][r]{(*)}%
\end{center}}

\newcounter{lsta}
\newenvironment{lsta}{%
\refstepcounter{lsta}%
\begin{center}
\begin{minipage}{.9\textwidth}}{%
\end{minipage}%
\makebox[.1\textwidth][r]{(**)}%
\end{center}}

\title[Transversality for Interval Translation Maps]{Transversality for Interval Translation Maps}
\author[Kostiantyn Drach]{Kostiantyn Drach}
\author[Leon Staresinic]{Leon Staresinic}
\author[Sebastian van Strien]{Sebastian van Strien}

\address{Universitat de Barcelona (Gran Via de les Corts Catalanes, 585, 08007 Barcelona, Spain)}
\address{Centre de Recerca Matem\`atica (Edifici C, Carrer de l'Albareda, 08193 Bellaterra, Spain)}
\email{kostiantyn.drach@ub.edu}

\address{University of Z\"urich (190 Winterthurerstrasse, 8057 Z\"urich, Switzerland)}
\email{leon.staresinic@math.uzh.ch}

\address{Imperial College London (180 Queen's Gate, South Kensington, London SW7 2AZ, UK)}
\email{s.van-strien@imperial.ac.uk}

\thanks{The first author was partially supported from grants CNS2025-166633 (AEI), PID2023-147252NB-I00 (AEI), CEX2020-001084-M (Maria de Maeztu Excellence program), and the ERC Advanced Grant ``SPERIG'' (\#885707). The second author acknowledges the support from Imperial College London through the Roth PhD scholarship and the Swiss National Science foundation through grant number TMCG-2\_213663·2023. The third author acknowledges a partial sponsorship via Bj\"orn Winckler's Marie Curie postdoctoral fellowship \#743959.
\newline 
\indent The content of this paper is part of an earlier manuscript \cite{drach2025densitystableintervaltranslation}, which is available on arXiv. That manuscript has been split into a three-part series: the present paper, \cite{drach2026characterisationstabilityintervaltranslation}, and \cite{drach2026topologicalprevalencefinitetype}.}

\begin{document}

\begin{abstract} An \emph{interval translation map} ($\ITM$) is a piecewise translation $T \colon I \to I$ defined on a finite partition $I_1, \ldots, I_r$ of an interval $I$ into $r \ge 2$ subintervals. In contrast to classical interval exchange transformations ($\IET$s), we do not require that the images of these subintervals are disjoint; in particular, $\ITM$s are not assumed to be bijective. Thus, $\ITM$s provide a natural non-invertible generalisation of $\IET$s. 

In this paper, we prove a transversality theorem for a family of dynamically defined vector subspaces that encode the dynamics of a given $\ITM$. As a consequence, we establish a perturbation result that gives a precise control of the first return dynamics to subintervals in $I$, while preserving the remaining global dynamics of the system. Beyond their independent interest, these results are a key technical ingredient in the proof of the Characterisation of Stability of $\ITM$s in~\cite{drach2026characterisationstabilityintervaltranslation} and in the establishment of the topological version of the Boshernitzan--Kornfeld Conjecture in~\cite{drach2026topologicalprevalencefinitetype}.

\end{abstract}

\maketitle


\section{Introduction}

\emph{Interval translation maps} ($\ITM$s) are piecewise isometries of an interval that form a natural non-invertible generalisation of the classical interval exchange transformations ($\IET$s), obtained by removing the bijectivity requirement. Throughout the paper, we fix the interval $I := [0,1)$. An $\ITM$ on $r \ge 2$ intervals is a map $T \colon I \to I$ determined by the following data:

\begin{itemize}
    \item a partition defined by \emph{discontinuity points} 
    $0 < \beta_1 < \beta_2 < \dots < \beta_{r-1} < 1$, and
    \item \emph{translation parameters} $\gamma_1, \dots, \gamma_r$.
\end{itemize}

Setting $\beta_0=0$ and $\beta_r=1$, the map acts by translation on each subinterval:
\[
T(x) = x + \gamma_i 
\quad \text{for } x \in [\beta_{i-1}, \beta_i),
\quad 1 \le i \le r.
\]
The condition $T(I) \subset I$ imposes the constraints $\gamma_i \in [-\beta_{i-1},\, 1-\beta_i]$ for all $i$. Hence, the parameter space $\ITM(r)$ of $\ITM$s on $r$ intervals naturally identifies with a convex polytope in $\R^{2r-1}$; we equip $\ITM(r)$ with the subspace topology.

A fundamental distinction from $\IET$s is that a typical $T \in \ITM(r)$ is not surjective; in general,
\[
T(I) \subsetneq I.
\]
This produces a nested sequence of forward images
\[
I \supsetneq T(I) \supseteq T^2(I) \supseteq T^3(I) \supseteq \dots.
\]
For $n \ge 0$ we define
\[
X_n := T^n(I), 
\qquad 
X := \bigcap_{n=0}^{\infty} X_n.
\]
Each $X_n$ is a finite union of intervals, and therefore $X$ is a non-empty \emph{attractor}. This leads to a fundamental dichotomy: an $\ITM$ is said to be of \emph{finite type} if $X_{n+1}=X_n = X$ for some $n$, and of \emph{infinite type} if $X_{n+1} \subsetneq X_n$ for all $n$.

The first explicit example of an infinite type map was constructed by Boshernitzan and Kornfeld in~\cite{MR1356616} for $r=3$. They conjectured that such behaviour should be rare:

\begin{mconjecture}
\label{conj:inf-type-zero}
For every $r \ge 2$, the set of infinite type $\ITM$s on $r$ intervals 
has Lebesgue measure zero in $\ITM(r)$.
\end{mconjecture}

This conjecture has driven much of the research on $\ITM$s, yet it remains widely open except in a few special cases. 
The first progress was made for $r=3$ within a special two-parameter family of $\ITM$s, where the conjecture was proved by Bruin and Troubetzkoy in~\cite{MR2013352}. 
This family admits a renormalization scheme analogous to the Rauzy--Veech induction for $\IET$s. The dynamical properties of this renormalization is a key ingredient in the proof that the set of infinite type maps has zero measure, just as Rauzy--Veech induction is a key ingredient in establishing that almost every minimal $\IET$ is uniquely ergodic (~\cite{MR644018,MR644019}) and weakly mixing (\cite{MR2299743}). Generalising the results for Rauzy-Veech induction needed in the proofs of these theorems, generic unique ergodicity and weak mixing for infinite type maps in the family from ~\cite{MR2013352} were established in ~\cite{MR4973368} and ~\cite{artigiani2026typicalweakmixingexceptional}, respectively (see also \cite{bruin2023interval}).

A related renormalization approach was later developed for $r \le 4$ within a broader class of \emph{double rotations}, introduced in~\cite{MR2152403} and further studied in~\cite{MR2966738,MR4397159}. 
Eventually, the conjecture was established in full for $r=3$ in~\cite{MR3124735} by showing that almost every $\ITM$ on three intervals can be renormalized to a double rotation. 
However, beyond these cases, there has been little progress toward a general renormalization theory for $\ITM$s with an arbitrary number of intervals, apart from some special situations (see~\cite{MR2308208}).

The Boshernitzan--Kornfeld conjecture concerns measure-theoretic genericity in $\ITM(r)$. 
A complementary notion of genericity is topological: one seeks to prove that a given property holds on a dense or residual (Baire generic) subset of the parameter space. 
In dynamical systems, one of the principal tools for establishing topological genericity is perturbation theory. 
Roughly speaking, perturbation results allow one to modify the system arbitrarily slightly so as to realise prescribed local dynamical behaviour, while keeping the remaining global dynamics under control. 
By performing such modifications inductively, one obtains arbitrarily small perturbations that induce global dynamical changes, thereby proving the genericity of the required property.

This strategy has been successfully applied in many contexts. 
In~\cite{MR593819}, perturbations were used to show that a dense $G_\delta$ subset of irreducible $\IET$s is uniquely ergodic. 
In~\cite{MR1373945}, geometric perturbations established that an open and dense set of foliations on affine surfaces is of Morse--Smale type, implying that an open and dense set of \emph{affine} $\IET$s is Morse--Smale. 
In smooth dynamics, perturbation techniques underlie classical genericity results such as density of structural stability for flows on surfaces~\cite{MR101951,MR142859}, the $C^1$ Closing Lemma~\cite{MR226669}, existence of homoclinic points for symplectomorphisms~\cite{MR331435}, and characterisation of structural stability in the symplectic setting~\cite{MR455049}. 
More recently, perturbation arguments were used in~\cite{MR4887204} to show that the attractor of a generic piecewise contraction consists of periodic orbits.

The situation for $\ITM$s is substantially more rigid. 
Unlike in smooth dynamics, one cannot perturb the map locally while leaving it unchanged elsewhere: any change in the parameters $\gamma_i$ and $\beta_i$ affects the forward iterates of all points in $I$. 
This holds for $\IET$s as well, but in this case, it is easier to produce explicit desirable perturbations that keep the global dynamics under control. 
Moreover, $\ITM$s lack an associated geometric theory (e.g.\ the theory of translation flows on surfaces); in these geometric settings, the additional flexibility of two-dimensional objects often facilitates perturbation arguments, from which results for (affine or classical) $\IET$s can then be deduced. Finally, $\ITM$s have a fixed derivative equal to $1$ (another form of rigidity), and therefore do not support the perturbations used for contracting maps.

Our approach to perturbation theory for $\ITM$s is therefore necessarily different. 
Starting from a finite type map $T$, the objective is to modify the \emph{first return map} $R_J$ to a connected component $J$ (assumed to be an interval) of $X$, while keeping the rest of the dynamics under control. The map $R_J$ is a bijection, and its domain is equal to the entire interval $J$ (Lemma \ref{lem:rj-facts}). Moreover, $J$ is partitioned into finitely many maximal half-open subintervals, which we call \emph{continuity intervals of $R_J$}, such that no point in their interiors lands on a discontinuity of $T$ before returning to $J$.

We first show that the first return map $R_J$ to $J$ can be naturally encoded by certain \emph{dynamically defined vectors} (see Subsection~\ref{subsec:prod}). These vectors describe how the dynamical properties of $R_J$ vary under small perturbations. A prescribed dynamical modification of $R_J$ can be realised by an arbitrarily small perturbation precisely when the corresponding encoding vectors are linearly independent. Similarly, controlling the dynamics of $T$ \emph{outside} the orbit of $J$ requires independence of additional families of dynamically defined vectors.

The main result of this paper is Theorem~\ref{thm:lin-dep}, which characterises the only possible linear dependencies among these natural dynamically defined vectors. 
The full statement is somewhat technical and requires preparation; we therefore refer to Subsection~\ref{subsec:statement-lin-indep} for details. 
Two important consequences of this theorem, which we state below and which are of independent interest, are central to the proofs in the sequel papers on the Characterisation of Stability~\cite{drach2026characterisationstabilityintervaltranslation} and on the Topological Boshernitzan--Kornfeld Conjecture~\cite{drach2026topologicalprevalencefinitetype}.

The first consequence concerns the intersection of subspaces generated by vectors encoding the dynamics of the return map to an interval component $J$ of $X$. Each continuity interval $J_j \subset J$ of $R_J$ determines two collections of vectors, corresponding to dynamical events at its left and right endpoints (the discontinuities of $R_J$). These itineraries describe the dynamics within each such $J_j$. There are three types of such vectors, for each left and right discontinuity:
\begin{itemize}
    \item \emph{First landing vectors}, describing finite itineraries of discontinuity points of $R_J$ before the first time they land at discontinuities of $T$;
    \item \emph{Critical connection vectors}, describing finite itineraries of discontinuities of $T$ before they land at discontinuities of $T$;
    \item \emph{Return vectors}, describing the part of the itinerary of a discontinuity of $R_J$ between the last landing at a discontinuity of $T$ and the return to $J$. 
\end{itemize}
Note that there might be $0$ and at most $2r-4$ critical connection vectors, and that each discontinuity of $R_J$ has at least two vectors associated with it. For each continuity interval $J_j \subset J$ of $R_J$ we denote by $V^+_j$, respectively $V^-_j$, the subspaces spanned by the three types of vectors associated with the left, respectively right, endpoint of $J_j$. We call them \emph{the positive}, respectively \emph{negative}, \emph{itinerary subspaces of $J_j$}.

The following theorem is the first main result of the paper (for the complete statement, see Lemma~\ref{lem:lin-dep-rj}).

\begin{lin-dep-ret-vec}
Let an interval $J \subset I$ be a connected component of the attractor $X$, and let $J_j \subset J$, with $j \in \{1,\dots,N\}$, be the continuity intervals of the first return map $R_J$ to $J$. Then:
\begin{enumerate}
    \item For every $j$, the intersection $V^+_j \cap V^-_j$ is one-dimensional.
    \item For every $k \neq j$, the subspaces $V^\pm_k$ and $V^\pm_j$ are transversal, i.e.
    \[
    V^\pm_k \cap V^\pm_j = \{0\}
    \quad \text{for every choice of signs.}
    \]
\end{enumerate}
\end{lin-dep-ret-vec}

Informally, the theorem asserts that the dynamically defined subspaces intersect only in the minimal way forced by the dynamics and are otherwise transverse. 
This transversality phenomenon is the guiding principle of the paper and motivates its title. 
We refer to Subsections~\ref{subsec:statement-lin-indep} and~\ref{subsec:two-consequences} for further discussion.

The second consequence is the Perturbation Lemma \ref{lem:pert-lem} that allows for detailed control over the return map dynamics to a chosen interval while also keeping the remaining global dynamics under control.

\begin{perturbation-lemma}
Let $J \subset I$ be a connected component of the attractor $X$, and $R_J$ be the first return map to $J$ with continuity intervals $J_1, \ldots, J_N$. Then there exists $\eps_0 > 0$ (depending on $J$ and $T$) such that for every $\eps \in (0, \eps_0)$ the following holds:

For every perturbation $\tilde J = \sqcup_{j=1}^N \tilde J_j$ of the partition $J= \sqcup_{j=1}^N J_j$, and for every perturbation $\tilde I_j$, $j \in \{1, \ldots, N\}$, of the return intervals $R_J(J_j)$ so that 
\begin{itemize}
    \item each $\tilde J_j \subset I$ is within an $\eps$-neighbourhood of the corresponding $J_j$;
    \item each $\tilde I_j$ is within an $\eps$-neighbourhood of the corresponding interval $R_J(J_j)$ and $|\tilde I_j| = |\tilde J_j|$, 
\end{itemize}
there exists a perturbation $\tilde T \in \ITM(r)$ of $T$ with the following properties:
\begin{enumerate}
    \item \emph{(Proximity of perturbations)} The distance between $T$ and $\tilde T$ in $\ITM(r)$ goes to $0$ as $\eps \to 0$.
    \item \emph{(Well-defined return map)} There is a well-defined first-return map $R_{\tilde J}$ to $\tilde J$ under $\tilde T$ such that $\tilde J_j$ are exactly the continuity intervals of $R_{\tilde J}$.
    \item \emph{(Proximity of return intervals)} The intervals $\tilde I_j$ are exactly the return intervals $R_{\tilde J}(\tilde J_j)$.
    \item \emph{(Same itinerary)} For each $j$, the return time of $\tilde J_j$ to $\tilde J$ is the same as the return time of $J_j$ to $J$, and the intervals $J_j$, $\tilde J_j$ have the same itineraries through the continuity intervals of $T$, respectively $\tilde T$, until the return time to $J$ or to $\tilde J$, respectively. 
    \item \emph{(Control of critical connections)} 
    The perturbation $\tilde T$ can be further specified so that for every \emph{critical connection} of $T$, i.e.\ a pair of discontinuities $\beta_i$, $\beta_k$ such that $T^n(\beta_i) = \beta_k$ for some $n \ge 1$, such that  
    \begin{enumerate}
        \item either $\beta_i, \beta_k$ are in the orbit under $T$ of some interval $J_j$ until its return to $J$,
        \item or $\beta_i, \beta_k$ are not in the orbit of $J$ under $T$,
    \end{enumerate}
    the perturbed image $\tilde T^n(\tilde \beta_i)$ can be chosen to lie arbitrarily within $\eps$-neighbourhood of the corresponding perturbed $\tilde \beta_k$ (in (a), provided that the itinerary of $\tilde \beta_i$ up to time $ \le n$ is the same as that of $\beta_i$).      
\end{enumerate}

Finally, if $J$ is chosen to be a \emph{maximal periodic interval} in $X$, i.e.\ $J$ is the largest interval without discontinuities in the orbit of its interior for which $R_J$ is the identity, then there exist a perturbation $\tilde T \in \ITM(r)$ and an interval $\tilde J$ in an $\eps$-neighbourhood of $J$ so that the conclusions (1)--(5) hold with $N=1$.
\end{perturbation-lemma}

Besides these two important consequences, we believe that the proof of Theorem \ref{thm:lin-dep} is of independent interest. To the authors' knowledge, there is no analogous setting in the literature where it is proven that a set of vectors related to itineraries of points is linearly independent. Thus, there are no available techniques that can be adapted to the setting of $\ITM$s, so the proof requires developing novel tools. The main idea is to link the linear dependence of vectors to certain partitions of the interval $I$, which are then inductively refined using the dynamics of $T$. By pushing this refinement procedure as far as possible, we obtain the desired form of linear dependence. For a detailed overview of the proof, see Subsection \ref{subsec:lin-dep-strategy}.

The structure of the paper is as follows. In Section \ref{sec:preliminaries}, we cover some basic preliminary material on $\ITM$s. In Section \ref{sec:lin-indep}, we introduce the product notation, dynamically defined vectors, and give the full statement of the main Theorem \ref{thm:lin-dep}, which we then use to prove the two main consequences: Lemma \ref{lem:lin-dep-rj} (Transversality of Itinerary Subspaces) and Lemma \ref{lem:pert-lem} (First Return Map Perturbation Lemma). Finally, in Section \ref{sec:proof-of-lin-dep}, we prove Theorem \ref{thm:lin-dep}.


\section{Preliminaries}
\label{sec:preliminaries}

\subsection{Notation and conventions} 
\label{subsec:not-conv}
In this subsection, we recall some of the main conventions and notation for $\ITM$s. Since the maps we are considering are discontinuous, it will be useful to introduce the notion of \textit{signed points}. Consider the relation $\sim$ defined on all half-open intervals contained in $I$, such that $I_1 \sim I_2$ if and only if the left endpoint of $I_1$ is equal to the left endpoint of $I_2$. For a point $x \in [0,1]$, we define $x^+$ as the equivalence class of all intervals $[x,x+\epsilon)$, for $\epsilon > 0$ sufficiently small so that $[x,x+\epsilon) \in I$. We refer to points of this form as the $+$-type points. Analogously, we define $x^-$ as the equivalence class (under the relation of having the same right endpoint) of all intervals $[x-\epsilon,x)$, where $\epsilon$ is sufficiently small so that $[x-\epsilon,x) \in I$, and we refer to these points as the $-$-type points. We will also say that $x^+$ is the `$+$-part' of $x$ and that $x^-$ is the `$-$-part' of $x$. We use the term `geometric point' to refer to points contained in the geometric interval $I = [0,1)$ and the term `signed point' to refer to the equivalence class $x^+$ or $x^-$ of a point $x \in I$. We will simply use the term `point' if it does not matter what type we are considering.

We adopt the convention that the signed point $x^+$ is immediately to the right of $x$, while $x^-$ is immediately to the left of $x$. With this convention, the definition of intervals contained in $I$ can be extended to signed points in the following way. By $[a,b]$, where $a$ and $b$ can be signed points or geometric points, we will denote the set of all points (signed or geometric) between $a$ and $b$, including $a$ and $b$. For example, the set $[x^-,y^-]$ contains $x^-,x,x^+$ and $y^-$, but not $y$ and $y^+$, and we will write $x^-, x, x^+, y^- \in [x^-,y^-]$ and $y, y^+ \notin [x^-,y^-]$. Analogously, $(a,b)$ will denote the set of all points between $a$ and $b$, excluding $a$ and $b$. We also define the sets $[a,b)$ and $(a,b]$ in the obvious way. With this notation $0^+,1^-$ are elements of $[0,1)$, but $0^-, 1,$ and $1^+$ are not. We will refer to all of these sets as \textit{intervals}. Thus the term `intervals' will refer to these more general sets if we are dealing with signed points, and to subsets of $I$ if we are not. Moreover, we will use the convention that the intervals of $I$ are half-open and of the form $[a,b)$, unless stated otherwise. We will adopt the following useful notation:

\begin{definition}
\label{def:touching}
We say that a pair of signed points $(a,b)$ \textit{touches} (or \textit{is touching}) if the set $\{a,b\}$ is equal to $\{x^+, x^-\}$ for some point $x \in I$. In that case, we will write $a \sim b$.
\end{definition}

The definition of any map $T$ can be extended to signed points in the following way. For a point $x \in I$, let $z_1 := \lim_{y \downarrow x}T^m(y)$ and $z_2:= \lim_{y \uparrow x} T^m(y)$. Since $T$ is discontinuous, $z_1^+ \sim z_2^-$ does not necessarily hold. Thus it makes sense to define:

\begin{align*}
T^m(x^+) &:= z_1^+ \\
T^m(x^-) &:= z_2^-.
\end{align*} 
The definition of the limits of signed points is also straightforward. If $\lim_{i \to \infty} z_i = z$ for some sequence $(z_i) \in I$, then we can define $\lim_{i \to \infty} z_i^{\pm} := z^{\pm}$. Finally, the distance $d(a,b)$ between two signed points $a,b$ is defined as the distance between the geometric points in $I$ corresponding to $a$ and $b$.

We now define the \textit{critical set} $\mathcal{C}$ of $T$ as the following set of signed points:

\[
\mathcal{C} := \{\beta_1^-, \beta_1^+,\dots,\beta_{r-1}^-, \beta_{r-1}^+\}.
\]
We will refer to the elements of $\mathcal{C}$ as either the critical points or the discontinuities, depending on the context. We will sometimes use the labels $\beta_0^+ := 0^+$ and $\beta_r^- := 1^-$, but we do not consider them as critical points. We denote by $\mathcal{C}^+$ the set of all $+$-type points in $\mathcal{C}$, and by $\mathcal{C}^-$ the set of all $-$-\textit{type} points in $\mathcal{C}$. We will also continue to use the term `discontinuity' to refer to geometric points in $I$ at which the map $T \colon I \to I$ is discontinuous. If we want to highlight that we are not considering a signed point, we will use the term `geometric discontinuity'.

If a point $x \in I$ does not land on a discontinuity of $T$ up to some time $n > 0$, then for our purposes there is no difference between iterates $T^n(x^-), T^n(x)$ and $T^n(x^+)$. If a geometric point $x$ does land on a discontinuity of $T$ at some time $n$, then we will not consider iterates of $x$ for any time larger than $n$, but iterates of $x^+$ and $x^-$ instead. In particular, we do not consider iterates of any geometric discontinuity $\beta$, but only of $\beta^+$ and $\beta^-$. Thus, the iterates of some discontinuity are assumed to be signed.

Most of the time, we do not need to know the index of a discontinuity with respect to the order in $I$ nor whether the discontinuity is of $+$-type or $-$-type. That is why we will often use labels $\beta$, $\betas$ and $\betass$ to denote the discontinuities we are dealing with. If we care about the sign of a discontinuity, we will use the labels $\beta^+, \betas^-$, etc. In that case, we will use the same label without a sign to denote the geometric discontinuity in $I$ corresponding to the signed discontinuity, e.g.\ we denote by $\beta$ the geometric discontinuity such that $\beta^+$ is the $+$-part of $\beta$. We will use the notation $\text{ind}(\beta) \in \{1, \dots, r-1\}$ to mean the index of $\beta$ with respect to this order inside $I$.

The perturbation of (or a map sufficiently close to) some starting map $T$ will be denoted by $\Tilde{T}$. Many of the most important objects associated with a map $T$, e.g.\ critical points, the set $X$ and boundary points of intervals contained in $X$, will have well-defined continuations for sufficiently small perturbations of $T$. For any such object $Z$, we will denote its continuation by $\Tilde{Z}$, e.g.\ $\Tilde{\beta}^+$, $\Tilde{X}$, $\Tilde{J}$ and similar.

We associate to each point $x \in I$ its \textit{itinerary}. Here it does not matter if a point is signed or geometric. The itinerary of $x$ is an infinite sequence of integers $(i_0(x), i_1(x), \dots, i_n(x), \dots)$, with $1 \le i_n(x) \le r$, where $i_n(x) = s$ means that $T^n(x) \in I_s$. We will often use `itinerary up to time $n$' to refer to the first $n$ elements of this sequence. 

Finally, for a point $x$ (can be either signed or geometric), we will refer to the set $O(x) := \{ x, T(x), T^2(x), \dots \}$ as the \textit{$T$-orbit} of $x$. Moreover, if $S$ is a subset of $I$, we will refer to the set $O(S) := \bigcup_{x \in S} O(x)$ as the $T$-orbit of $S$. If the map $T$ is clear from the context, we will omit it and simply use the term `orbit'. We will refer to the set $O(x,n) := \{ x, \dots, T^{n-1}{x} \}$, where $n \ge 0$, as the orbit up to time $n$ of $x$, and to the set $O(S,n) := \bigcup_{x \in S} O(x,n)$ as the orbit up to time $n$ of $S$.

\subsection{Basic results and definitions}
\label{subsec:basic-res-def}
In this subsection, we recall the main definitions and prove some elementary facts for $\ITM$s. We start with the following useful definition:

\begin{definition}
\label{def:c1c2}
For $T \in \ITM(r)$, we define the following sets:
\begin{itemize}
    \item $\mathcal{C}_1$ is the set of all $\beta \in \mathcal{C}$ that eventually land on a discontinuity;
    \item $\mathcal{C}_2$ is the set of all $\beta \in \mathcal{C}$ that never land on a discontinuity, but are eventually periodic;
    \item $\mathcal{C}_0 = \mathcal{C}_1 \cup \mathcal{C}_2$;
    \item $\mathcal{C}_i^{\pm} = \mathcal{C}^{\pm} \cap \mathcal{C}_i$, for $i \in \{ 0, 1, 2 \}$.
\end{itemize}
\end{definition}

Periodic discontinuities have a distinguished role in Theorem \ref{thm:lin-dep}. Such discontinuities are a part of entire intervals of periodic points, called maximal periodic interval:

\begin{definition}
\label{def:max-per-int}
A maximal interval $J \subset I$ consisting of periodic points with the same itinerary is called a \textit{maximal periodic interval}. In this way, $J$ is a half-open interval and its boundary points have the property that they land on discontinuities of $T$ or on the boundary points of $I$. Moreover, no point in the interior of $J$ lands on a discontinuity. 
\end{definition}

Since there are $r-1$ discontinuities for a map $T \in \ITM(r)$, there can only be finitely many maximal periodic intervals with pairwise disjoint orbits.

The second type of intervals we are interested in are the connected components of $X$ that are equal to an interval. It will be convenient to use the following shorthand for such intervals:

\begin{definition}
\label{def:comp-int}
We will say that an interval $J$ is an \textit{interval component} of $X$ if $J$ is equal to a connected component of $X$.
\end{definition} 

Note that an interval component of $X$ can be a maximal periodic interval, but that a maximal periodic interval can be strictly contained in an interval component of $X$.

We now prove several elementary results about general $\ITM$s. The following is a fundamental classification of the dynamics of every point in $I$:

\begin{lemma}[Orbit Classification Lemma]
\label{lem:orb-class}
For each point $x \in I$, either signed or geometric, at least one of the following three possibilities holds:
\begin{enumerate}
\item (Precritical) $x$ lands on a discontinuity of $T$;
\item (Preperiodic) $x$ lands on a periodic point of $T$;
\item (Accumulation) $x$ accumulates on a discontinuity of $T$. More precisely, there exist discontinuities $\beta, \betas \in \mathcal{C}$ such that $x$ accumulates on $\beta$ from the left and on $\betas$ from the right.
\end{enumerate}
\end{lemma}

\begin{proof} Assume that $x \in I$ does not land on a critical point or on a periodic point. We will prove that there is a $\beta \in \mathcal{C}$ such that $x$ accumulates on it from the left, with the case of right-accumulation being analogous. Assume the contrary, that $x$ does not accumulate on any critical points from the left. This implies that the orbit of $x$ stays a distance $\delta > 0$ away to the left of every critical point of $T$. Because of the assumption that $x$ does not land on a periodic point, this implies that the orbit of $x$ also stays at least $\delta$ away from every periodic point of $T$. Indeed, if the point $T^n(x)$, for some $n \ge 0$, is less than $\delta$ away from some periodic point $y$ of minimal period $p$, then the closed interval between $T^n(x)$ and $y$ cannot map forward continuously until time $p$, because this would mean that $T^n(x)$ is periodic. Thus some iterate of this interval contains a discontinuity, which means that the orbit of $x$ gets less than $\delta$ close to $\mathcal{C}$, which is a contradiction.

Thus there exists a maximal open interval $J_0$ of length at least $\delta$ with $x$ in as its right boundary point, such that no point in this interval lands on a critical or a periodic point. Thus the interval $J_0$ maps forward continuously for all time, and every interval $T^i(J_0)$ has the property that none of its points land on critical or periodic points. Let $J_i$ be the maximal interval containing $T^i(x)$ such that no point in it lands on a critical or periodic point, and note that $T^i(J_0) \subset J_i$. Since all of the intervals $J_i$ have length at least $\delta$, there exists $n > m \ge 0$ such that $J_n \cap J_m \neq \emptyset$. By maximality, $J_n = J_m$ and thus $T^{n-m}(J_m) = J_m$. Since $T^{n-m}_{\vert J_m}$ is continuous and a translation, it must be the identity, so $T^{n-m}(T^m(x)) = T^m(x)$, which is a contradiction.
\end{proof} 

Let $A_1'$ (we choose this notation since we will use $A_1$ in the more important result Lemma \ref{lem:x-structure}) be the union of all interval components of $X$. The following is a simple lemma about the restriction of $T$ to $A_1'$.

\begin{lemma}
\label{lem:t-bij-on-A1'}
$T(A_1') \subset A_1'$ and $T_{\vert{A_1'}}$ is a bijection. In particular, the inverse $T^{-1}_{\vert A_1'} : A_1' \to A_1'$ is well defined.
\end{lemma}

\begin{proof}
Let $J$ be any component interval contained in $A_1'$. $J$ is partitioned into a finite number (possibly only $1$) of subintervals $J_1, \dots, J_N$ on which $T$ is continuous. Note that each of these subintervals is non-trivial, i.e.\ has positive length. Since $T(X) = X$, the image of $T(J_i)$, for every $1 \le i \le N$, is contained in $X$, and therefore also in a component interval of $X$. Thus $T(J) \subset A_1'$, so the first claim follows.

We first prove that $T_{\vert{A_1'}}$ is surjective. Let $y$ be any point in $A_1'$. Since $y$ is contained in a component interval of $X$, for $\delta > 0$ sufficiently small, there is an interval of points $J = [y,y+\delta)$ contained in the same component interval of $X$ as $y$ which does not contain $T(\beta)$ for every $\beta \in \mathcal{C}$. By construction, for all $n \ge 1$, $J \subset X_n$ and there exists at least one interval $J_{n-1} \subset X_{n-1}$ such that $T(J_{n-1}) = J$, and such that $J_{n-1}$ does not contain a discontinuity of $T$. By letting $n \to \infty$, we see that there exists such an interval $J_{\infty}$ contained in $X$ as well. Since $J_{\infty}$ has length $\delta$ by construction and its left endpoint maps to $y$, we see that $y$ has a $T$-preimage in $A_1'$.  

Finally, assume now that $T_{\vert{A_1'}}$ is not injective, i.e.\ that there is a point $y \in A_1'$ which has two preimages $x_1, x_2 \in A_1'$. Since $x_1, x_2 \in A_1'$, there exists a $\delta > 0$ such that the intervals $[x_1, x_1 +\delta)$ and $[x_2, x_2 +\delta)$ are contained in $A_1'$. Since $X_n$ forms a descending sequence of sets contained in $X$, the sequence of differences in their measures $|X_n \setminus X_{n+1}|$ must form a Cauchy sequence. Thus $|X_n \setminus X_{n+1}| < \delta/2$ for $n$ large enough. But $[x_1, x_1 +\delta)$ and $[x_2, x_2 +\delta)$ are both contained in $X_n$ and have the same image, which means that $|X_n \setminus X_{n+1}|$ is at least $\delta$, and this gives the desired contradiction.
\end{proof}

First return maps to intervals contained in $X$ are one of the central objects in this paper. 
\begin{definition}[First return map]
\label{def:rj}
For an interval $J \subset X$, define $R_J$ to be the \emph{first return map} (or simply the \emph{return map}) to $J$ under $T$, i.e.\ for every $x \in J$ that returns to $J$, we define $R_J(x) := T^k(x)$ for the smallest integer $k = k(x) \ge 1$ such that $T^k(x) \in J$.
\end{definition}
Note that $R_J$, under this definition, may be defined on a strict subset of $J$. However, the following simple lemma shows that this is not the case for intervals in $X$:
\begin{lemma}
\label{lem:rj-facts}
The domain of $R_J$ is the entire interval $J$. $R_J: J \to J$ is bijective and $J$ is partitioned into finitely many maximal half-open subintervals such that no point in their interiors lands on a discontinuity of $T$.
\end{lemma}

The return map $R_J$ is continuous on these subintervals of $J$, so for simplicity we refer to these intervals as `continuity intervals of $R_J$'. We usually use $x$ and $y$ to denote the boundary intervals of $J$, so that $J = [x,y)$.

We use the following notation for such interval components $J$. By Lemma \ref{lem:rj-facts}, there are finitely many points $a_{1}, \dots, a_{N-1}$ in the interior of $J$ that land on discontinuities before returning to $J$. It is also convenient to define $a_{0} := x$ and $a_{N} := y$. The continuity intervals of $R_J$ are denoted by $J_1, \dots, J_N$. The return time of $J_j$ to $J$ is denoted by $r_j$, for $1 \le j \le N$, and the landing times of $a_j$ to discontinuities of $T$ are denoted by $l_j$, for $1 \le j \le N-1$.

It is sufficient to prove Lemma \ref{lem:rj-facts} for interval components of $X$. Indeed, this would imply that the return map to each such component is an $\IET$. Since every interval $J$ contained in $X$ must be contained in such an interval component, the result for $J$ follows for the classical result for $\IET$s (see \cite{MR0516048} for the proof). For the purposes of the proof, we use a slightly different notation for $R_J$ than introduced above, as this notation only makes sense after proving Lemma \ref{lem:rj-facts}.

\begin{proof}[Proof of Lemma \ref{lem:rj-facts}]
As discussed, we prove the case when $J$ is an interval component of $X$. Let $\mathcal{C}(J)$ be the set of discontinuities of $T$ for which there exists a point in $J$ that land on it. Let $b_1 < \dots < b_N \subset J$ be the finite set of points in $J$ that have the smallest landing times to discontinuities in $\mathcal{C}(J)$, i.e.\ for every $\beta \in \mathcal{C}(J)$ there exists exactly one $b_i$, with $1 \le i \le N$, such that $b_i$ is the point in $J$ with the smallest $l_i$ such that $T^{l_i}(b_i) = \beta$. Note that the points $b_i$ are well defined, as $T_{\vert A_1}$ is a bijection, by Lemma \ref{lem:t-bij-on-A1'}.

The points $b_i$ induce a partition of $J$ into finitely many subintervals. We may assume that $b_1$ is contained in the interior of $J$, with the case of $b_1$ in the left boundary point being analogous Thus there are exactly $N+1$ such intervals, and we denote them by $J_1, \dots, J_{N+1}$. We claim that the points in each interval $J_i$ of this partition all return to $J$ at the same finite time, and have the same itinerary up to this time.

Let $J_i$ be arbitrary, with $1 \le i \le N+1$. Let $m > 0$ be the smallest time (which clearly exists) such that at least one of the following three things happens:

\begin{enumerate}
    \item $T^m(J_i)$ contains a discontinuity $\beta$ in its interior,
    \item $T^m(J_i) \cap T^n(J_i) \neq \emptyset$ for some $0 < n < m$, or
    \item $T^m(J_i) \cap J \neq \emptyset$.
\end{enumerate}

Assume that (1) happens first, and let $x$ be the point in $J_i$ that lands on $\beta$ at time $m$. Then by definition $\beta \in \mathcal{C}(J)$, and there exists an $b_j$ such that $b_j$ has the smallest landing time $l$ to $\beta$ out of all points in $J$. In particular, $l < m$, since $T^{l}_{\vert A_1'} : A_1' \to A_1'$ is a bijection by Lemma \ref{lem:t-bij-on-A1'}. Since $T^{-1}_{\vert A_1'}$ is well defined and bijective by Lemma \ref{lem:t-bij-on-A1'}, this means that $T^{m-l}(x) = a_j$, so (3) must have happened at a strictly smaller time than (1), contradicting our choice of $m$.

Now assume that (1) does not happen at time $m$, but that (2) does. Thus the iterates $T^j(J_i)$ do not contain a discontinuity of $T$ for all $0 \le j \le m$. This means that the intervals $T^j(J_i)$ have a continuous $T^{-1}_{\vert A_1'}$-images of every $1 \le j \le m$. Thus in particular $T^{n-1}(J_i) \cap T^{m-1}(J_i) \neq \emptyset$, contradicting our choice of $m$.

From this it follows that only (3) happens at time $m$. Thus $J_i$ maps forward continuously up to time $m$, so $T^m(J_i)$ is a single interval. Then if $T^m(J_i) \setminus J \neq \emptyset$, the interval $J$ can be extended to include $T^m(J_i) \setminus J \neq \emptyset$, contradicting the fact that it is a connected component of $X$. Thus $T^m(J_i) \subset J$. This implies that the domain of the first return map $R_J$ to $J$ is the entire interval $J$, and that the $J_i$ subintervals are the continuity intervals for $R_J$. If $R_J$ is not injective, i.e.\ if $R_J(J_i)$ and $R_J(J_k)$ intersect for some $i \neq k$, then similarly as in case (2) above, we can continuously pull-back the $R_J(J_i) \cup R_J(J_k)$ with $T^{-1}$ to get that one of these intervals must have intersected $J$ earlier, contradicting the definition of the first return map.
\end{proof}

Using Lemma \ref{lem:rj-facts}, we may prove the following structure result for the closure $\overline{X}$ of the attractor $X$. Recall that we consider the closure of $X$ in $[0,1]$, not in the relative topology of $I$.

\begin{lemma}
\label{lem:x-structure}
For every map $T \in \ITM(r)$, the closure $\overline{X} \subset [0,1]$ is equal to $A_1 \cup A_2$, where $A_1$ is a finite union of closed intervals and $A_2$ is a Cantor set. The union $A_1 \cup A_2$ is disjoint, except possibly at the right endpoints of the intervals in $A_1$. Moreover, one of the sets $A_1$ and $A_2$ may be empty.
\end{lemma}

Let us remark that this result has been mentioned in several places in the literature (\cite{MR1356616}, \cite{MR2013352}, \cite{MR2308208}), but no explicit proof is available.

\begin{proof}[Proof of Lemma \ref{lem:x-structure}]
Let $A_1^{'}$ be the union of all maximal intervals contained in $X$. We first prove that $A_1^{'}$ always contains at most finitely many intervals. Indeed, we claim that $A_1^{'}$ is equal to the union of orbits of intervals of $X$ that contain discontinuities of $T$. More precisely, $A_1^{'} = \bigcup_{\beta \in X \cap \, \mathcal{C}} O(J_{\beta})$, where $J_{\beta}$ is the interval component of $X$ containing $\beta \in X$. Since the return map $R_J$ to any interval $J$ of $X$ always has finitely many continuity intervals, this proves the claim. Let $J$ be some interval of $X$. By maximality, at least one boundary point of any continuity interval $J'$ of $R_J$ lands on a discontinuity in $X$ before returning to $J$. Thus each such interval $J'$ lands into an interval $J_{\beta}$, for some $\beta \in X$, before returning to $J$. This means that $J'$ is contained in the orbit of $J_{\beta}$, since the return map $R_{J'}$ is well-defined. Thus the entire interval $J$ is in $A_1^{'}$, which proves the claim. We set $A_1 := \overline{A_1^{'}}$.

Let $A_2 := \overline{ \overline{X} \setminus A_1}$, and assume that it is non-empty. Recall the well-known characterisation of a topological Cantor set: it is a compact, metrizable and totally disconnected set with no isolated points. We now show $A_2$ satisfies all of these properties. Compactness follows by definition, while metrizability follows from it being a subset of $[0,1]$. It is totally disconnected because the image of any path between two different points in $A_2$ is an interval in $A_2$, which is impossible as $\overline{X} \setminus A_1$ is nowhere dense. Indeed, if it were dense in some interval, then this entire interval would be contained in $X_n$ for all $n$, as $X_n$ is a finite union of intervals. This gives that this interval is in $X$ as well, so we get a contradiction with the definition of $A_2$.

Finally, we prove that there are no isolated points by showing that for any point $x \in A_2$, there is a sequence in $A_2$ accumulating on $x$. First of all, $x$ is not periodic, as it would then be contained in a maximal periodic interval. As it is contained in $X$, the set of all of its preimages is therefore infinite and contained in $A_2$. Thus its set of preimages has an accumulation point $y$ in $A_2$. If the orbit of $y$ accumulates on $x$, we are done, so assume that this is not the case. This means that there exists a strictly monotone sequence $x_{n_k} \to y$ of $n_k$-preimages, with $n_1 < n_2 \dots$, such that the itinerary of $x_{n_k}$ and $y$ up to time $n_k$ is different, i.e.\ the iterates of the interval $[x_{n_k},y)$ before time $n_k$ contain a discontinuity. Since there are finitely many discontinuities, there is a discontinuity $\beta$ that occurs for infinitely many $k$. Thus the orbit of $y$ accumulates on $\beta$, so $\beta$ is contained in $A_2$ and the orbit of $\beta$ accumulates on $x$. Indeed, since $y^-$ is not periodic (preimages of $x$ would then be contained in a periodic interval), for any finite time, the iterates of $y$ have positive distance between discontinuities. Thus the times at which the interval $[x_{n_k},y)$ lands on $\beta$ are unbounded, so $x$ is not isolated.  
\end{proof}

Finally, Lemma \ref{lem:J-dynamics} gives some elementary dynamical properties for the interval components of $X$, and it will be used implicitly.

\begin{lemma}
\label{lem:J-dynamics}  
Let $T$ be an $\ITM$. Then the following two properties hold: 
\begin{enumerate}[label=(\alph*)]
    \item Every interval component of $X$ is of the form $[T^{k_1}(\beta^+), T^{k_2}(\betas^-)]$, respectively, for some $\beta^+, \betas^- \in X \cap \mathcal{C}$ and $k_1,k_2 \ge 0$;
    \item If $T^{l_1}(\beta)$ is an interior point of an interval component of $X$ for some $\beta \in X \cap \mathcal{C}$ and $l_1 \ge 0$, then there exist $\betas \in X \cap \mathcal{C}$ and $l_2 \ge 0$ such that $T^{l_1}(\beta) \sim T^{l_2}(\betas)$.
\end{enumerate}
\end{lemma}

Recall that $T_{\vert A_1'}$ is a bijection by Lemma \ref{lem:t-bij-on-A1'}, so the inverse $T^{-1}_{\vert A_1'}$ is well defined.

\begin{proof} 
Observe that the $T^{-1}_{\vert A_1'}$-preimage of every boundary point of an interval component of $X$ is either a boundary point of an interval component of $X$ or a discontinuity of $T$. Fix $x$ in the boundary of $X$ and assume that $T^{-n}(x) \notin \mathcal{C}$ for every $n \ge 0$. Since there are finitely many boundary points, this means that $T^{-n}(x)$ is eventually periodic. Because $T_{\vert A_1'}$ is a bijection, this means that $x$ is also periodic. By assumption, it does not land on a discontinuity. Thus there is an interval of points around $x$ that are all periodic with the same period. This contradicts $x$ being a boundary point of $X$, as all periodic points are contained in $X$. Thus $x = T^n(\beta)$ for some discontinuity $\beta \in \mathcal{C}$, which proves (a).

For (b), let $x$ be a point that is either a discontinuity contained in an interval component of $X$ or is a boundary point of an interval component of $X$. If the image of $x$ is contained in the interior of an interval component of $X$, then there is a point $y$ which is either a discontinuity or a boundary point of $X$ such that $T(x) \sim T(y)$. By part (a), $T(y)$ is of the form $T^{l_2}(\betas)$ in both cases. 

Assume that $T^{l_1}(\beta) \in \text{int}(X)$ for some $n>0$, $\beta \in X \cap \mathcal{C}$. We may assume $\beta$ is the last discontinuity in $O(\beta, l_1)$. Thus $T^i(\beta) \notin \mathcal{C}$ for $0 < i < l_1$. We may further assume that $T^i(\beta) \notin \text{int}(X)$ for $0 < i < l_1$. Indeed, if some iterate $T^i(\beta)$, for $0 < i < l_1$, was in $\text{int}(X)$ and satisfied (b), then every iterate $T^j(\beta)$, for $i < j \le l_1$, is in $\text{int}(X)$ and also satisfies (b). This is because $\beta$ does not land on a discontinuity before time $l_1$. This means that $T^{l_1-1}(\beta)$ is in the boundary of an interval component of $X$, so (b) once again follows from the paragraph above. 
\end{proof}

\section{Linear Dependence of Dynamical Vectors}
\label{sec:lin-indep}  

\subsection{Product notation and perturbations}
\label{subsec:prod}

In this subsection, we introduce the product notation, which will be used throughout the paper. It is a very convenient way to represent and control iterates of critical points.

For each  $T$,  $x \in I$, $s=1,\dots,r$ and $n\ge 1$, we define:
\[
k_s(x,n,T) := \# \{ T^i(x) \in I_s; 0\le i < n \}.
\]
Thus $k_s(x,n,T)$ represents the number of entries of $x$ into $I_s$ up to time $n$. This is an analogue of the first term in the continued fraction expansion for circle rotations, see \cite[Section 1.1]{MR1239171}. When the map $T$ is clear from the context, we will usually simply write $k_s(x,n)$.

Let $W(r) := \R{}^{r} \bigoplus \R{}^{r-1}$ be the $(2r-1)$-dimensional real vector space that contains the \textit{coefficient vectors} (introduced below), and let $(\bm{e}_1, \dots, \bm{e}_r)$ and $(\bm{f}_1, \dots, \bm{f}_{r-1})$ be the canonical bases for $\R{}^{r}$ and $\R{}^{r-1}$, respectively. Recall that $\ITM(r)$ is the parameter space of $\ITM$s on $r$ intervals, and that it is a convex polytope contained in $\mathbb{R}^{2r-1}$. We call the elements of this space \textit{parameter vectors}. These vectors also have canonical coordinates coming from the ambient space $\mathbb{R}^{2r-1}$. We will use the shorthand $(\gamma \, \beta)$ for a parameter vector $(\gamma_1 \dots \gamma_r \, \beta_1 \dots \beta_{r-1})$.

Let $\langle \cdot, \cdot \rangle$ be the standard scalar product on $\mathbb{R}^{2r-1}$. Since $W(r) = \R{}^{r} \bigoplus \R{}^{r-1}$ and $\ITM(r)$ is a subset of $\mathbb{R}^{2r-1}$, it makes sense to write $\langle v, (\gamma \, \beta) \rangle =  \sum_{s=1}^r v_s \gamma_s + \sum_{s=1}^{r-1} v_{s+r} \beta_s$ for a coefficient vector $v = \sum_{s=1}^r v_s \bm{e}_s + \sum_{s=1}^{r-1} v_{s+r} \bm{f}_s \in W(r)$ and a parameter vector $(\gamma \, \beta) \in \ITM(r)$. We call $\langle v, (\gamma \, \beta) \rangle$ the \textit{product} of $v$ and $(\gamma \, \beta)$.

The first reason for introducing this product is to obtain a formula for any iterate $T^n(\betas)$ of some discontinuity $\betas$. Define the vector of coefficients in $v(\betas,n,T) \in W(r)$ of this iterate as:

\begin{equation}
\label{vec:v-beta-n}
v(\betas,n,T) \coloneqq \left(\sum_{1 \le s \le r} k_s(\betas, n, T) \, \bm{e}_s, \,  \bm{f}_{\text{ind}(\betas)}\right).
\end{equation}
Then the following holds:
\begin{align*}
T^n(\betas) &= \sum_{s=1}^r k_s(\betas,n,T) \gamma_s + \betas \\
&= \langle \, v(\betas,n,T) , (\gamma \, \beta) \, \rangle,
\end{align*}
where the vector $(\gamma \, \beta)$ corresponds to the defining coefficients of $T$. With this notation, we can associate to any landing of some discontinuity $\betas$ on some other discontinuity $\betass$ a coefficient vector in the following way. If $T^n(\betas) = \betass$, then the landing vector $L$ is defined as:

\[
L \coloneqq \left(\sum_{s = 1}^r \, k_s(\betas, n) \, \bm{e}_s, \, \bm{f}_{\text{ind}(\betas)} - \bm{f}_{\text{ind}(\betass)} \right).
\]
Using the above formula for the iterate $T^n(\betas)$, we get:

\begin{align*}
0 &= T^n(\betas) - \betass \\
&= \sum_{s=1}^r k_s(\betas,n,T) \gamma_s + \betas - \betass \\
&= \langle \, L , (\gamma \, \beta) \, \rangle.
\end{align*}

The second reason why the scalar product notation is useful is that it gives explicit formulas for how the $T$-iterates of points change after a perturbation $\delta$ of $T$. Let $\{ \bm{g_1}, \dots, \bm{g_r}, \bm{h_1}, \dots, \bm{h_{r-1}} \}$ be the canonical basis of the tangent space $\bm{T}_T \ITM(r)$ at point $T = (\gamma \, \beta)$ in $\ITM(r)$. Each perturbation $\Tilde{T}$ of a map $T$ in $\ITM(r)$ is given by a linear functional $\bm{\delta}: \bm{T}_T \ITM(r) \to \R$, so that $\Tilde{\gamma}_s = \gamma_s + \bm{\delta}(\bm{g_s})$ and $\Tilde{\beta}_s = \beta_s + \bm{\delta}(\bm{h_s})$. Thus $\Tilde{T}$ is obtained from $T$ by perturbing in the direction $\delta := (\bm{\delta}(\bm{g}_1) \, \dots \, \bm{\delta}(\bm{g}_r) \, \bm{\delta}(\bm{h_1}) \, \dots \, \bm{\delta}(\bm{h}_{r-1})) \in \bm{T}_T \ITM(r)$. Hence, after identifying $\bm{T}_T \ITM(r)$ and $\mathbb{R}^{2r-1}$, it makes sense to write $(\gamma \, \beta) + \delta = (\Tilde{\gamma}\, \Tilde{\beta})$.

If the itineraries of $\beta$ under $T$ and $\Tilde{\beta}$ under $\Tilde{T}$ are the same up to some finite time $n \ge 1$, then we have the following formula: 

\begin{align*}
\Tilde{T}^n(\Tilde{\beta}) - T^n(\beta) &= \sum_{s = 1}^r k_s(\beta,n) \Tilde{\gamma}_s + \Tilde{\beta} - \sum_{s = 1}^r k_s(\beta,n) \gamma_s - \beta \\
&= \sum_{s = 1}^r k_s(\beta,n) (\Tilde{\gamma}_s - \gamma_s) + \Tilde{\beta} - \beta \\
&= \sum_{s = 1}^r k_s(\beta,n) \bm{\delta}(\bm{g_s}) + \bm{\delta}(\bm{h_{\text{ind}(\beta)}}) \\
&= \langle \, v(\beta,n,T) , \delta \,  \rangle.
\end{align*}
Thus if we want $T^n(\betas) = \betass$ to hold after perturbation as well, the following must hold (assuming the itinerary of $\betas$ up to time $n$ does not change):

\[
\langle \, L , \delta \, \rangle = 0.
\]
These connections between vectors and perturbations are what allow us to use Theorem \ref{thm:lin-dep} to prove the Perturbation Lemma \ref{lem:pert-lem}. See subsection \ref{subsec:two-consequences} for details.

\subsection{Dynamical vectors and statement of Theorem \ref{thm:lin-dep}}

\label{subsec:statement-lin-indep}

The crucial property that allows us to change one part of the dynamics while keeping the other parts as they are is the linear independence of itinerary vectors of discontinuities. This is the content of Theorem \ref{thm:lin-dep} and Corollary \ref{cor:lin-indep}. To state these results precisely, we need some preliminary results and notation. When introducing objects that depend on some interval $J$, the dependence will be explicit in the notation if the object in question will be used in the context where there are several $J$-intervals, and it will be implicit otherwise.

Let $J = [x^{J,+},y^{J,-}]$ be an interval component of $X$. Recall that the return map $R_J$ to $J$ is well-defined (see Definition \ref{def:rj} and the paragraph below) and that there are finitely many points $a_{1}^{J}, \dots, a_{N_J-1}^J$ in the interior of $J$ that land on discontinuities before returning to $J$. Let $J_1^J, \dots, J_{N_J-1}^J$ be the maximal intervals in $J$ such that no point in their interior lands on a discontinuity, and let $r_j^J$ be the return time of $J_j^J$ to $J$. Let $a_{0}^{J,+} := x^+$ and $a_{N_J}^{J,-} := y^-$ be the boundary points of $J$. For each $1 \le j \le N_J-1$, let $m_{j}^{J,+}$ be the number of discontinuities that $a_{j}^{J,+}$ lands on before returning to $J$ and, for $1 \le k \le m_{j}^{J,+}$, let $\beta^{J,+}(j,k)$ be the $k$-th discontinuity along the orbit up to return time to $J$ of $a_j^{J,+}$. Define $m_{j}^{J,-}$ and $\beta^{J,-}(j,k)$ similarly, for $1 \le j \le N_J-1$ and $1 \le k \le m_{j}^{J,-}$. For each $1 \le j \le N_J$, the points $\beta^{J,+}(j,1)$ and $\beta^{J,-}(j,1)$ are the $+$ and $-$ part of a single discontinuity, which we denote by $\beta^J(j)$. Figure \ref{fig:points-in-J-orbit} illustrates these points for an interval $J$ for which $R_J$ has three continuity intervals and $\sigma = (3 2 1)$.

\begin{figure}
    \centering
    \includegraphics[width=\linewidth]{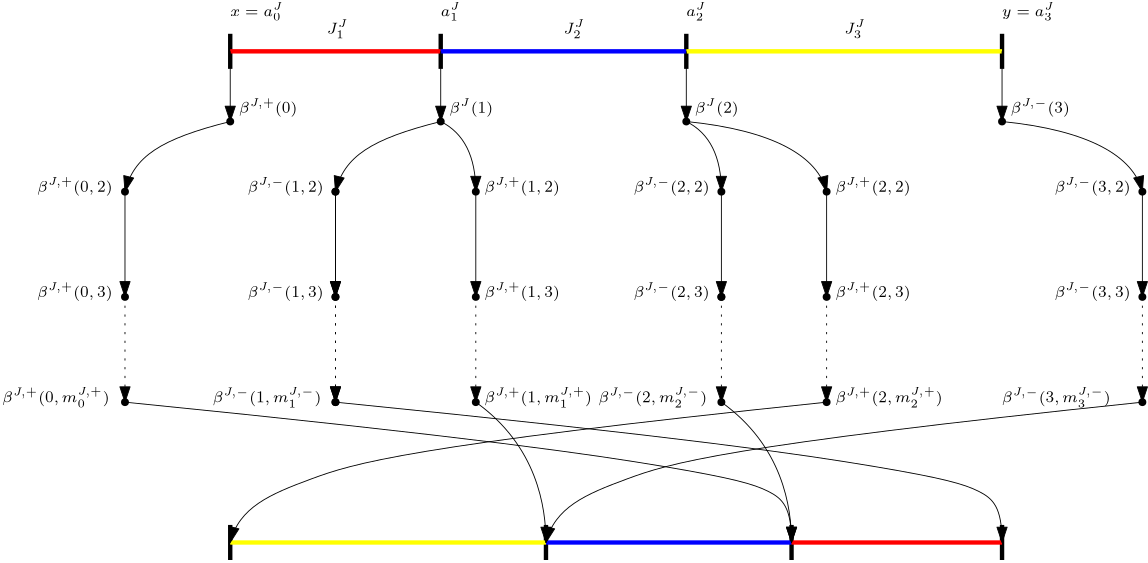}
    \caption{All dynamically important points in the orbit of an interval $J$ for which the return map $R_J$ has three continuity intervals.}
    \label{fig:points-in-J-orbit}
\end{figure}

Recall that $W(r) = \R{}^{r} \oplus \R{}^{r-1}$ is the space of coefficient vectors, and that $(\bm{e}_1, \dots, \bm{e}_r)$ and $(\bm{f}_1, \dots, \bm{f}_{r-1})$ are the canonical bases for $\R{}^{r}$ and $\R{}^{r-1}$, respectively.

We now define three types of dynamically defined vectors associated to the orbit of $J$. The first ones are the \textit{first landing vectors}, that correspond to the first time the points $a_1^J, a_2^J, \dots, a_{N_J-1}^J$ land on discontinuities of $T$:

\begin{definition}[First Landing vectors]
\label{def:lan-vec}
For $1 \le j \le N_J-1$, let $l_{j}^J$ be the landing time of $a_{j}^J$ to $\beta^J(j)$ and let $L_{j}^{J} \in W(r)$ be the associated vector, called the \emph{first landing vector}:

\[
L_{j}^{J} \coloneqq \left(\sum_{s = 1}^r k_s(a_j^J, l_{j}^J) \, \bm{e}_s, \, - \bm{f}_{\text{ind}(\beta^J(j))}\right).
\]
\end{definition}

Recall from subsection \ref{subsec:not-conv} that ind$(\beta)$ is the index of $\beta$ with respect to the order of critical points inside of $I$. If $a_j$ is a discontinuity of $T$, i.e.\ if the landing time $l_j^J$ is zero, then $L_j^{J} = \left( 0, - \bm{f}_{\text{ind}(\beta^J(j))}\right)$ by definition. We call $L_j^{J}$ the first landing vector of $a_j^J$ to $\beta^J(j)$ because the following holds:

\begin{align}
\label{eq:L-vec-property}
\begin{split}
0 &= a_j^J + \sum_{s = 1}^r k_s(a_j^J, l_{j}^J) \gamma_s - \beta^J(j) \\
&= \left(\sum_{s = 1}^r k_s(a_j^J, l_{j}^J) \gamma_s - \beta^J(j)\right) + a_j^J \\
&= \langle L_j^{J}, (\gamma \, \beta) \rangle + a_j^J.
\end{split}
\end{align}
\noindent
The second type of vectors we need to consider are the \textit{critical connection vectors}, which correspond to landings of discontinuities in the orbit of $J$ onto other discontinuities (which are thus also in the orbit of $J$).

\begin{definition}[Critical connection vectors]
\label{def:cc-vec}
For each $1 \le j \le N_J-1$ and $1 \le k < m_{j}^{J,+}$, let $q^{J,+}(j,k)$ be the landing time of $\beta^{J,+}(j,k)$ to $\beta^{J,+}(j,k+1)$ and let $C^{J,+}(j,k) \in W(r)$ be the associated vector, called the \emph{critical connection vector}:

\[
C^{J,+}(j,k) \coloneqq \left(\sum_{s=1}^r k_s(\beta^{J,+}(j,k), q^{J,+}(j,k)) \, \bm{e}_s, \, \bm{f}_{\text{ind}(\beta^{J,+}(j,k))} - \bm{f}_{\text{ind}(\beta^{J,+}(j,k+1))}\right).
\]
\end{definition}

The following holds by construction:
\begin{align}
\label{eq:C-vec-property}
\begin{split}
0 &= \beta^{J,+}(j,k) + \sum_{s=1}^r k_s(\beta^{J,+}(j,k), q^{J,+}(j,k)) \gamma_s - \beta^{J,+}(j,k+1) \\
&= \sum_{s=1}^r k_s(\beta^{J,+}(j,k), q^{J,+}(j,k)) \gamma_s + \beta^{J,+}(j,k) - \beta^{J,+}(j,k+1) \\
&= \langle C^{J,+}(j,k), (\gamma \, \beta) \rangle.
\end{split}
\end{align}
Define $C^{J,-}(j,k) \in W(r)$ analogously for $1 \le j \le N_J-1$ and $1 \le k < m_{j}^{J,-}$.

The third type of dynamical vector we need to consider are the \textit{return vectors}, which correspond to the return of points $a_1^{J,-}, a_1^{J,+}, \dots, a_{N_J-1}^{J,+}, a_{N_J-1}^{J,-}$ to the interval $J$.

\begin{definition}[Return vectors]
\label{def:ret-vec}
For each $1 \le j \le N_J-1$, let $r_{j}^{J,+}$ be the time at which $\beta^+(j,m_{j}^{J,+})$ lands into $J$ and let $R_{j}^{J,+} \in W(r)$ be the associated vector, called the \emph{return vector}:

\[
R_{j}^{J,+} \coloneqq \left(\sum_{s=1}^r k_s(\beta^{J,+}(j,m_{j}^{J,+}), r_{j}^{J,+}) \, \bm{e}_s, \, \bm{f}_{\text{ind}(\beta^{J,+}(j,m_{j}^{J,+}))}\right).
\]
\end{definition}
By definition, the following holds:
\[
\langle R_j^{J,+}, (\gamma \, \beta) \rangle = \sum_{s=1}^r k_s(\beta^{J,+}(j,m_{j}^{J,+}), r_{j}^{J,+}) \gamma_s + \beta^+(j,m_{j}^{J,+}) \in J.
\]
Note that also:
\begin{equation}
\label{eq:R-vec-property}
R_J(a_{j}^+) = \langle R_j^{+}, (\gamma \, \beta) \rangle    
\end{equation}
since $a_{j}$ lands on $\beta^+(j,m_{j}^{+})$. Define $r_{j}^{J,-}$ and $R_{j}^{J,-} \in W(r)$ analogously for $1 \le j \le N_J-1$.

For the boundary points $x^{J,+} = a_0^{J,+}$ and $y^- = a^{J,-}_{N_J}$, the definitions of the corresponding dynamical vectors depend on whether they land on discontinuities of $T$ before returning to $J$ or not. If $a_0^{J,+}$ lands on a discontinuity before returning to $J$, then we may analogously as for other $a_j^{J,+}$, where $1 \le j \le N_J-1$, define the vectors $L_0^J$, $C^{J,+}(0,k)$ and $R^{J,+}_0$. In the case when $a_0^{J,+}$ does not land on a discontinuity of $T$ before returning to $J$, we only define the return vector $R_0^{J,+}$ to $J$:
\begin{equation}
\label{eq:alter-R-vec}
R_0^{J,+} := \left(\sum_{s=1}^r k_s(a_0^{J,+}, r_{0}^{J,+}) \, \bm{e}_s, \, 0 \right),
\end{equation}
where $r_0^+$ is the return time of $a_0^+$ to $J$. Analogously, if $a_{N_J}^{J,-}$ lands on a discontinuity of $T$, we may define the vectors $L_{N_J}^J, C^{J,-}(N_J,k)$ and $R^{J,-}_{N_J}$. If $a_{N_J}^{J,-}$ does not land on a discontinuity of $T$ before returning to $J$, then define $R_{N_J}^{J,-}$ analogously as for $a_0^{J,+}$. 

\begin{figure}
    \centering
    \includegraphics[width=\linewidth]{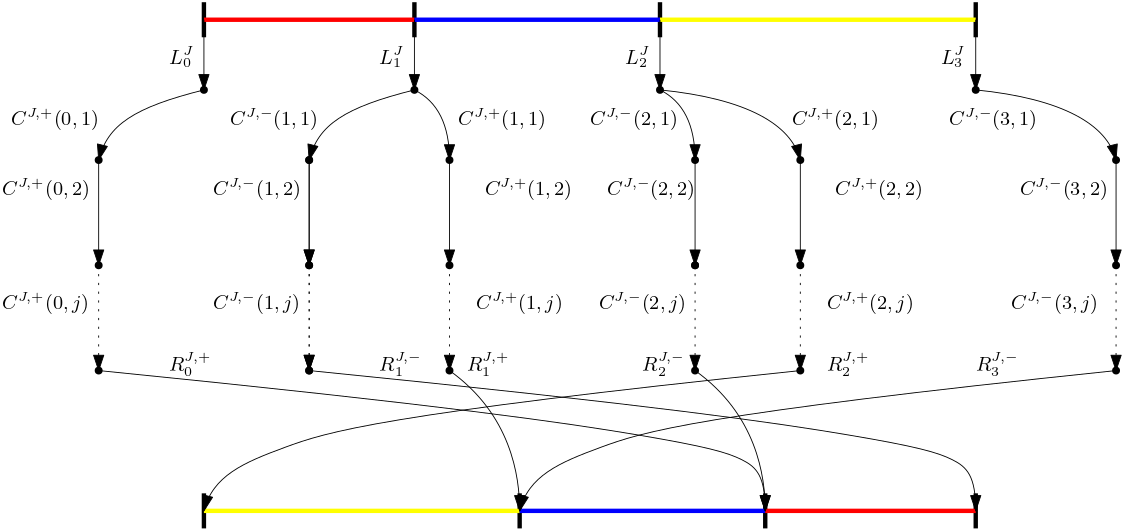}
    \caption{Vectors associated to the orbit of an interval $J$ for which the return map $R_J$ has three continuity intervals.}
    \label{fig:vectors-in-J-orbit}
\end{figure}

We would also like to have control over the dynamics of the discontinuities not contained in $X$. For this purpose, we also introduce the following more general notation for a critical connection vector. Recall from subsection \ref{subsec:basic-res-def} that we say that $\beta \in \mathcal{C}_1$ if $\beta$ eventually lands on a discontinuity. Let $\beta'$ denote the first such discontinuity in the orbit of $\beta$, and let $q(\beta)$ be the landing time of $\beta$ to $\beta'$. Note that $\beta$ and $\beta'$ are signed discontinuities, by the definition of $\mathcal{C}_1$.

\begin{definition}[General critical connection vectors]
Let $\beta$ be a discontinuity of $T$ which eventually lands on another discontinuity. Define the vector $C_{\beta} \in W(r)$ as:

\[
C_{\beta} := \left(\sum_{s=1}^r k_s(\beta,q(\beta))\bm{e}_s, \, \bm{f}_{\text{ind}(\beta)} - \bm{f}_{\text{ind}(\beta')}\right),
\]
\end{definition}
By definition, we have that:

\begin{align*}
0 &= \beta + \sum_{s=1}^r k_s(\beta,q(\beta)) \gamma_s - \beta' \\
&= \langle C_{\beta}, (\gamma \, \beta) \rangle
\end{align*}
Recall also from \ref{subsec:basic-res-def} that $\mathcal{C}_2$ denotes the set of all discontinuities that never land on other discontinuities, but are eventually periodic. The dynamically important vector related to these discontinuities is the landing vector of such a discontinuity into $X$. We do not include these vectors in Theorem \ref{thm:lin-dep}. The reason is that they land into interiors of intervals which we can control with Theorem \ref{thm:lin-dep}. Thus we can already control the dynamics of these discontinuities with this theorem, without the need to include them in the statement.

In the statement of Theorem \ref{thm:lin-dep}, there are two types of intervals contained in $X$. There is a distinguished interval $J_0$, which can be any interval contained in $X$, and there are intervals $J_1, \dots, J_n$ that are maximal periodic intervals.

To ease the notation, we will not explicitly state the ranges for all indices that appear in the statement, as they are the same as in the discussion above. Moreover, we will abbreviate the dependence of vectors and numbers on an interval $J_i$, for $0 \le i \le n$, to the dependence on the number $i$, e.g.\ we will write $C^{i,+}(j,k)$ instead of $C^{J_i,+}(j,k)$.

For the maximal periodic intervals $J_1, \dots, J_n$, the only points that land on discontinuities are the boundary points, so $N_{i}= 1$ for $1 \le i \le n$. For each $J_i$, we will consider the vector $C^{i,+}(0,m_0^{i,+}) := R^{i,+}_0 + L^{i}_0$, instead of the first landing vector $L^{i}_0$ and the return vector $R^{i,+}_0$. The reason for this is that the vector $C^{i,+}(0,m_0^{i,+})$ is equal to the critical connection vector corresponding to the landing of $\beta^{i,+}(0,m_0^{i,+}-1)$ to $\beta^{i,+}(0,1)$. This is the special property of maximal periodic intervals: their return and first landing vectors combine into a critical connection vector, which is not true in general. This will be crucially used in the proof of Proposition \ref{prop:refine}, which is one of the main steps in the proof of Theorem \ref{thm:lin-dep}. We define $C^{i,-}(1,m_1^{i,-})$ analogously. 

Before finally stating Theorem \ref{thm:lin-dep}, let us give an interpretation of its statement. The sums of all vectors associated with the left and right endpoints of a continuity interval of a return map to an interval component $J_0$ of $X$ have to be equal because the itineraries of these points up to the return time to $J_0$ are equal (see Figure \ref{fig:vectors-in-J-orbit}). Thus there is an obvious linear dependence between these vectors: we may take all coefficients of vectors associated with the left endpoint equal to some $\alpha$ and all coefficients of vectors associated with the right endpoint to $-\alpha$. Theorem \ref{thm:lin-dep} states, in particular, that this obvious linear dependence is in fact the only one possible. More importantly, and perhaps surprisingly, Theorem \ref{thm:lin-dep} also states that the same conclusion that there is only one possible linear dependence between these vectors continues to hold if we enlarge this set of vectors by adding \textit{all} of the vectors associated to the return map to $J_0$ (as in Figure \ref{fig:vectors-in-J-orbit}), as well as all vectors associated to maximal periodic intervals of $T$ and discontinuities outside of $X$.

\begin{theorem}[Linear Dependence of Dynamical Vectors]
\label{thm:lin-dep}
Let $T \in \ITM(r)$ be an interval translation map. Let $J_0, J_1, J_2, \dots, J_n$ be the intervals in $X$ that have pairwise disjoint orbits, and let $\mathcal{C}_{\neg X}$ be the set of all discontinuities in $\mathcal{C}_1$ that are not contained in $X$. Assume that $J_0$ is an interval component of $X$ and that $J_1, \dots, J_n$ are maximal periodic intervals. As above, let the following be the vectors related to their dynamics:

\begin{lstv}
\label{eq:vectors}
\begin{itemize}
    \item $L^{0}_j, R^{0,+}_j, R^{0,-}_j, C^{i,+}(j,k), C^{i,-}(j,k)$;
    \item $C_{\beta},\text{ for } \beta \in \mathcal{C}_{\neg X}$.
\end{itemize}
\end{lstv}
Assume that there exist real coefficients

\begin{lsta}
\label{eq:a-coefficients}
\begin{itemize}
    \item $\alpha_{j}^{0}, \alpha_j^{0,+}, \alpha_j^{0,-}, \alpha^{i,+}(j,k), \alpha^{i,-}(j,k)$;
    \item $\alpha_{\beta}, \text{ for } \beta \in \mathcal{C}_{\neg X}$.
\end{itemize}
\end{lsta}
such that the vectors in \hyperref[eq:vectors]{(*)} satisfy:

\begin{align}
\label{eq:lin-dep-sum}
\begin{split}
&\sum_{i=1}^n \left( \sum_{k=1}^{m_0^{i,+}} \alpha^{i,+}(0,k) C^{i,+}(0,k) \right) + \sum_{i=1}^n \left( \sum_{k=1}^{m_1^{i,-}} \alpha^{i,-}(1,k) C^{i,-}(1,k) \right) \\
+&\sum_{j=0}^{N_0-1} \left( \sum_{k=1}^{m_j^{0,+}-1} \alpha^{0,+}(j,k) C^{0,+}(j,k) + \alpha_j^{0,+} R_j^{0,+} \right) \\ 
+&\sum_{j=1}^{N_0} \left( \sum_{k=1}^{m_j^{0,-}-1} \alpha^{0,-}(j,k) C^{0,-}(j,k) + \alpha_j^{0,-} R_j^{0,-} \right) \\
+& \left( \sum_{j=0}^{N_0} \alpha_j^0 L^0_j \right) + \sum_{\beta \in \mathcal{C}_{\neg X}} \alpha_{\beta} {C}_{\beta} = 0.    
\end{split}
\end{align}
Then the following equalities hold:
\begin{align}
\label{eq:lin-dep-equality1}
\begin{split}
&\alpha^{i,+}(0,1) = \dots = \alpha^{i,+}(0,m_0^{i,+}) = -\alpha^{i,-}(1,1) = \dots = - \alpha^{i,-}(1,m_{1}^{i,-}); \\
&\alpha^{0,+}(j,1) = \dots = \alpha^{0,+}(j,m_j^{0,+}-1) = \alpha_j^{0,+} \\
&=-\alpha^{0,-}(j+1,1) = \dots = - \alpha^{0,-}(j+1,m_{j+1}^{0,-}-1) = -\alpha_{j+1}^{0,-},
\end{split}
\end{align}
for all $1 \le i \le n$ and $0 \le j \le N_0-1$, respectively. Moreover:
\begin{align}
\label{eq:lin-dep-equality2}
\begin{split}
& \alpha_0^0 = \alpha^{0,+}(0,1); \\
& \alpha_{N_0}^0 = \alpha^{0,-}(N_0,1); \\
& \alpha_j^0 = \alpha^{0,+}(j,1) + \alpha^{0,-}(j,1); \\
& \alpha_{\beta} = 0,
\end{split}
\end{align}
for all $1 \le j \le N_0-1$ and for all $\beta \in \mathcal{C}_{\neg X}$, respectively.
\end{theorem}

\begin{figure}
    \centering
    \includegraphics[width=\linewidth]{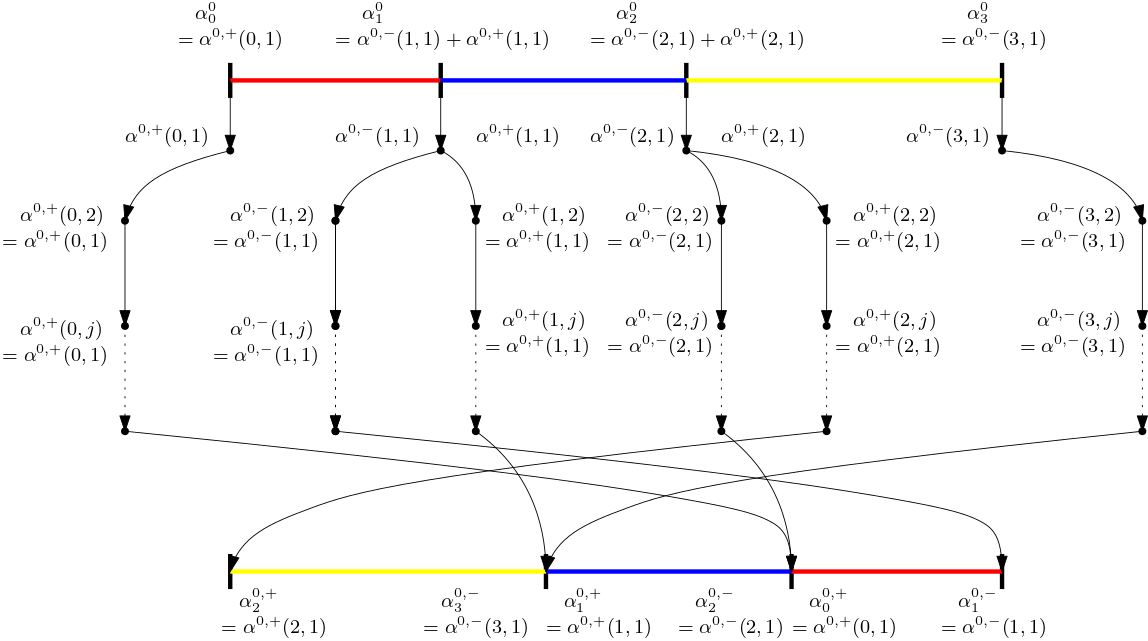}
    \caption{The coefficients for the first landing vectors and critical connection vectors are at the start of their corresponding arrows, while the coefficients for the return vectors are at the end. The indicated equalities between the coefficients come from the conclusion of Theorem \ref{thm:lin-dep}.}
    \label{fig:return-map-vectors}
\end{figure}


The conclusion of Theorem \ref{thm:lin-dep} for $J_0$ (the same interval as in Figures \ref{fig:points-in-J-orbit} and \ref{fig:vectors-in-J-orbit}) with a return map that has three continuity intervals is demonstrated in Figure \ref{fig:return-map-vectors}. The assumption $J_1, \dots, J_n$ are maximal periodic intervals is necessary for our proof to work. A natural question to ask is whether Theorem \ref{thm:lin-dep} holds in a bigger generality:

\begin{question}
\label{q:general-lin-dep}
Does Theorem \ref{thm:lin-dep} hold if we assume that $J_1, \dots, J_n$ are interval components of $X$, and not just maximal periodic intervals? 
\end{question}

\begin{remark}
\label{remark:diff-thm-statement}
Recall from \eqref{eq:alter-R-vec} that in the case when the boundary points of $J_0$ do not land on discontinuities on $T$, there are no first landing and critical connection vectors associated with them, and the definition of the return vector changes. We have chosen to state Theorem \ref{thm:lin-dep} under the assumption that both boundary points of $J$ land on a discontinuity before returning to $J$, as this is the more complicated case since we have more vectors to deal with. Indeed, if $a^{0,+}_0$ does not land on a discontinuity before returning to $J_0$, then we may simply omit the vectors $L^0_0$ and $C^{0,+}(0,k)$ from the statement of Theorem \ref{thm:lin-dep}, with the proof and the conclusion staying the same. Analogously for the case when $a^{0,-}_{N_0}$ does not land on a discontinuity before returning to $J_0$.
\end{remark}

The proof of Theorem \ref{thm:lin-dep} is divided over Subsections \ref{subsec:refining-part}, \ref{subsec:refining-part} and \ref{subsec:proof-of-lin-dep}, while in Subsection \ref{subsec:lin-dep-strategy} we give an overview of the proof. The important consequence of Theorem \ref{thm:lin-dep} is the following corollary:

\begin{corollary}[Linear Independence of Dynamical Vectors]
\label{cor:lin-indep}
Let $T \in \ITM(r)$ be an interval translation map. Let $J_0, J_1, J_2 \dots, J_n$ be intervals contained in $X$ with pairwise disjoint orbits. Assume that $J_0$ is an interval component of $X$ and that $J_1, \dots, J_n$ are maximal periodic intervals. Then the vectors

\begin{itemize}
    \item $L_{j}^0, R^{0,+}_j, C^{0,+}(j,k), C^{0,-}(j,k)$;
    \item $C^{i,+}(0,k)$ for $1 \le k \le m_0^{i,+}$ and $C^{i,-}(1,k)$ for $1 \le k < m_0^{i,-}$;
    \item $C_{\beta}$, for $\beta \in \mathcal{C}_{\neg X}$,
\end{itemize}
form a linearly independent set.
\end{corollary}

\begin{proof}
For the sake of contradiction, assume that they are linearly dependent with coefficients:
\begin{itemize}
    \item $\alpha_{j}^{0}, \alpha_j^{0,+}, \alpha^{0,+}(j,k), \alpha^{0,-}(j,k)$;
    \item $\alpha^{i,+}(0,k)$ for $1 \le k \le m_0^{i,+}$ and $\alpha^{i,-}(1,k)$ for $1 \le k < m_0^{i,-}$;
    \item $\alpha_{\beta}, \text{ for } \beta \in \mathcal{C}_{\neg X}$.
\end{itemize}
This linear dependence can be extended to include the vectors $R^{0,-}_j$ and $C^{i,-}(1,m_0^{i,-})$ by setting:
\begin{itemize}
    \item $\alpha_{j}^{0,-} = 0$;
    \item $C^{i,-}(1,m_0^{i,-}) = 0$,
\end{itemize}
for all $1 \le j \le N_0$ and $1 \le i \le n$, respectively. Then by equation \eqref{eq:lin-dep-equality1} from Theorem \ref{thm:lin-dep}, we have that:
\begin{align*}
&\alpha^{i,+}(0,1) = \dots = \alpha^{i,+}(0,m_0^{i,+}) = -\alpha^{i,-}(1,1) = \dots = - \alpha^{i,-}(1,m_{1}^{i,-}) = 0; \\
&\alpha^{0,+}(j,1) = \dots = \alpha^{0,+}(j,m_j^{0,+}-1) = \alpha_j^{0,+} \\
&=-\alpha^{0,-}(j+1,1) = \dots = - \alpha^{0,-}(j+1,m_{j+1}^{0,-}-1) = -\alpha_{j+1}^{0,-} = 0,
\end{align*}
for all $1 \le i \le n$ and $0 \le j \le N_0-1$, respectively. Thus by equation \eqref{eq:lin-dep-equality2} from Theorem \ref{thm:lin-dep}, we also have that:
\begin{align*}
& \alpha_0^0 = \alpha^{0,+}(0,1) = 0; \\
& \alpha_{N_0}^0 = \alpha^{0,-}(N_0,1) = 0; \\
& \alpha_j^0 = \alpha^{0,+}(j,1) + \alpha^{0,-}(j,1) = 0 + 0 = 0; \\
& \alpha_{\beta} = 0,
\end{align*}
for all $1 \le j \le N_0-1$ and for all $\beta \in \mathcal{C}_{\neg X}$, respectively. Thus all of the coefficients of the linear dependence are equal to zero, which is what we wanted to prove.
\end{proof}

\begin{remark}
\label{remark:renormalization}
As mentioned in the introduction, currently, there does not exist a general renormalization scheme for $\ITM$s. For $\IET$s, the renormalization matrix for the basic step of the Rauzy--Veech induction is invertible (see \cite{MR2219821}). The same is true for the Zorich acceleration of the Rauzy--Veech induction, as the matrices associated with this acceleration are the products of matrices associated to the basic steps of the Rauzy--Veech induction (\cite{MR2219821}). Thus the column/row vectors of these matrices are linearly independent. In both cases, the renormalization matrix comes from considering the first return map to an appropriately chosen subinterval of $I$. With this in mind, Corollary \ref{cor:lin-indep} and Theorem \ref{thm:lin-dep} can be viewed as more general versions of these facts, as we are again considering subintervals of $I$ and the return maps to these intervals, as well as the vectors associated to these return maps. The difference is that they, as currently defined, do not arise from a renormalization scheme. A potential strategy for finding a renormalization scheme for $\ITM$s would thus be to find a more general version of Theorem \ref{thm:lin-dep} that does not require the return map to a specially chosen interval to be bijective, and then use the corresponding dynamical vectors to construct a matrix. This matrix would then hopefully point to what the appropriate renormalization scheme should be.
\end{remark}

\subsection{Two consequences of Theorem \ref{thm:lin-dep}}
\label{subsec:two-consequences}
In this section, we derive two consequences of Theorem \ref{thm:lin-dep} described in the introduction. The first one, Lemma \ref{lem:lin-dep-rj} is obtained by simply restricting the conclusion of Theorem \ref{thm:lin-dep} only to the vectors related to the orbit of the first interval $J_0$. In the statement, we omit the dependence of the vectors and coefficients on the interval $J_0$, namely the superscript $0$, as there is now only a single interval. The other notation is the same as in the statement of Theorem \ref{thm:lin-dep}.

Recall from the introduction that for each $1 \le j \le N$, $V_j^+$ and $V_j^-$ denote the positive and negative itinerary subspaces of $J_j$, respectively.
Thus for every $1 \le j \le N$, $V_j^+$ is defined as the span of vectors $L_{j-1}$, $C^+(j-1,k)$ and $R_{j-1}^+$ and $V_j^-$ as the span of vectors $L_{j}$, $C^-(j,k)$ and $R_{j}^-$. 

\begin{lemma}[Linear Dependence of Return Map Vectors]
\label{lem:lin-dep-rj}
Let $T \in \ITM(r)$, and let $J$ be an interval component of $X$. Assume that there exist real coefficients $\alpha_{j}, \alpha_j^{+}, \alpha_j^{-}, \alpha^{+}(j,k), \alpha^{-}(j,k)$ such that:

\begin{align}
\label{eq:lin-dep-sum}
\begin{split}
&\sum_{j=0}^{N-1} \left( \sum_{k=1}^{m_j^{+}-1} \alpha^{+}(j,k) C^{+}(j,k) + \alpha_j^{+} R_j^{+} \right) \\ 
+&\sum_{j=1}^{N} \left( \sum_{k=1}^{m_j^{-}-1} \alpha^{-}(j,k) C^{-}(j,k) + \alpha_j^{-} R_j^{-} \right) \\
+& \left( \sum_{j=0}^{N} \alpha_j L_j \right) = 0.    
\end{split}
\end{align}
Then the following equalities hold:
\begin{align}
\label{eq:lin-dep-equality1}
\begin{split}
&\alpha^{+}(j,1) = \dots = \alpha^{+}(j,m_j^{+}-1) = \alpha_j^{+} \\
&=-\alpha^{-}(j+1,1) = \dots = - \alpha^{-}(j+1,m_{j+1}^{-}-1) = -\alpha_{j+1}^{-},
\end{split}
\end{align}
for all $0 \le j < N$. Moreover, $\alpha_j = \alpha^-(j,1) + \alpha^+(j,1)$ for all $1 \le j < N$.

In particular, if we denote by $\Sigma_j$ the vector $(L_{j-1} + \sum_{k=1}^{m_j^{+}-1} C^+(j-1,k) + R^+_{j-1}) = (L_{j} + \sum_{k=1}^{m_j^{-}-1} C^-(j,k) + R^-_{j})$, then $V_j^+$ and $V_j^-$ intersect only along the line $\Sigma_j \mathbb{R} $, while $V^{\pm}_j$ and $V^{\pm}_k$ for $j \neq k$ intersect only at $\{0\}$ for all choices of signs.
\end{lemma} \qed

In the case when the boundary point $a_0^+$ of $J$ does not land on a discontinuity of $T$, we need to remove $L_0$ and $C^+(0,k)$ from \eqref{eq:lin-dep-sum} and all of the coefficients for $j=0$ from the first line of \eqref{eq:lin-dep-equality1} except $\alpha_0^+$. A similar procedure should be done when $a_N^-$ does not land on a discontinuity.

The second result is a dynamical formulation of Corollary \ref{cor:lin-indep}: the Perturbation Lemma \ref{lem:pert-lem}. This is a very useful statement, because it allows us to make explicit perturbations without having to refer to the heavy vector and product notation. To state it, we need the following useful definition:

\begin{definition}
\label{def:trans-fac}
Let $n > 0$ be a natural number. We define the \textit{translation factor} $Tr(x,n)$ at time $n$ of a point $x \in I$ as the translation factor of $T^n$ restricted to $x$, i.e.\ $Tr(x,n) = T^n(x)-x$.
\end{definition}

This definition can easily be extended to intervals $K \subset I$ of points that have the same itinerary up to time $n > 0$, by calling the translation factor $Tr(K,n)$ at time $n$ of $K$ the translation factor of any point $x$ in $K$. Note that $Tr(K,n) = 0$ if and only if $K$ is periodic with period $n$. We denote the translation factors defined for a perturbation $\Tilde{T}$ of $T$ by $\Tilde{T}r$.

The Perturbation Lemma \ref{lem:pert-lem} has two parts, which can be roughly described as follows. The first states that all sufficiently small dynamical changes (of a certain type) of a return map $R_J$ to an interval component $J$ of $X$ are achievable by perturbations of the underlying map $T$ that depend continuously on the size of the dynamical changes. The second states that the perturbation can also be chosen so that we control two types of critical connections: those in the orbit of $J$ up to return time and those outside the orbit of $J$.

\begin{lemma}[Perturbation lemma]
\label{lem:pert-lem}
Let $T$ be an interval translation map such that $T(I)$ is compactly contained in the interior of $I$. Let $J$ be one of the following: an interval component of $X$ or a maximal periodic interval. Then there exists an $\epsilon_0 > 0$ depending on $J$ and $T$, with the following properties. For every $\epsilon < \epsilon_0$ and every choice of $\epsilon^{\gamma}_1, \dots \epsilon^{\gamma}_{N}, \epsilon^{\beta}_0, \dots \epsilon^{\beta}_{N} \in (-\epsilon, \epsilon)$ there exists a perturbation $\Tilde{T}$ such that $|\Tilde{T} - T| \to 0$ as $\epsilon \to 0$, with $|\cdot|$ being the distance in $\ITM(r)$, and the following holds:

\begin{enumerate}
    \item There exists an interval $\Tilde{J} \subset I$ that is $\epsilon$-close to $J$ and partitioned into intervals $\Tilde{J}_j = [\Tilde{a}_{j-1},\Tilde{a}_j)$, with $1 \le j \le N$, such that $\Tilde{J}_j$ maps forward continuously up to time $r_j$ under the iterates of $\Tilde{T}$ and has the same itinerary up to time $r_j$ as $J_j$ for all $1 \le j \le N$. In the case when $J$ is an interval component of $X$, we may set $\Tilde{a_j} - a_j = \epsilon^{\beta}_j$ for all $0 \le j \le N$;
    \item The difference between the translation factors of $\Tilde{J}_j$ and $J_j$ is $\epsilon^{\gamma}_j$ for all $1 \le j \le N$, i.e.\ $\Tilde{T}r(\Tilde{J}_j,r_j) - Tr(J_j,r_j) = \epsilon^{\gamma}_j$.
\end{enumerate}
Moreover, we have that:

\begin{enumerate}[label=(\alph*)]
    \item Let $T^n(\beta^+) = \betas^+$ be a critical connection such that $\beta^+$ and $T^n(\beta^+)$ are both contained in the $T$-orbit of $J_j$ up to time $r_j$ for some $1 \le j \le N$. We may assume that $\Tilde{\beta}^+$ is still contained in the $\Tilde{T}$-orbit of $\Tilde{J}_j$ up to time $r_j$ and the difference $\Tilde{T}^n(\Tilde{\beta}^+) - \Tilde{\beta}_*^+$ can be chosen arbitrary in $[0,\epsilon)$. Analogously for critical connections $T^n(\beta^-) = \betas^-$ and the difference $ \Tilde{\beta}_*^- - \Tilde{T}^n(\Tilde{\beta}^-)$;
    \item For every critical connection $T^n(\beta) = \betas$, with $\beta,\betas \notin O(J)$, such that either $\beta \notin X$, or $\beta, \betas$ are part of a single periodic orbit, the difference $\tilde T^n(\tilde \beta) - \tilde \betas$ can be chosen arbitrary in $(-\epsilon,\epsilon)$.
\end{enumerate}
\end{lemma}

In (a), the changes are such that the itineraries of the critical points are preserved, while this is not necessarily the case in part (b). The assumption that $T(I)$ is compactly contained in the interior of $I$ holds for every $\ITM$ in a complement of finitely many hyperplanes, so this may be assumed without loss of generality, and it allows us not to include several necessary assumptions on the $\epsilon, \epsilon^{\gamma}_1, \dots \epsilon^{\gamma}_{N}, \epsilon^{\beta}_0, \dots \epsilon^{\beta}_{N}$ that guarantee that the perturbed map still has image contained in $[0,1)$. 

\begin{remark}
    Note that the statement of Lemma~\ref{lem:pert-lem} is slightly more general than the one stated in the introduction. In particular, the intervals $\tilde J_j$ need not return to $\tilde J$ at time $r_j$, i.e.\ the set $\tilde T^{r_j}(\tilde J_j) \setminus \tilde J$ might be non-empty. For sufficiently small perturbations, this can only happen for $j = \tau(1)$ or $j = \tau(N)$.   
\end{remark}

\begin{proof}
Assume first that $J$ is an interval component of $X$. In the setup of Corollary \ref{cor:lin-indep} we set $J_0 := J$, and $J_1, \dots, J_n$ equal to maximal periodic intervals with pairwise disjoint orbits that contain all of the periodic discontinuities of $X$ not contained in the orbit of $J$. Assume first that the boundary points $a_0^{0,+}$ and $a_{N_0}^{0,-}$ land on discontinuities, so that the corresponding first landing vectors are defined. 

We now explain how to choose $\epsilon_0$. Let $M$ be the maximum among the following times: the return times of the intervals $J^{0,j}$ to $J_0$, the minimal periods of the intervals $J_1, \dots, J_n$, and the landing times of points in $\mathcal{C}_{\neg X}$ to $X$. Let $v(\mathcal{C},T)$ be the union of the set of dynamical vectors corresponding to the iterates of all discontinuities of $T$ up to time $M$. More precisely, for every $\beta \in \mathcal{C}$ and every iterate $T^n{\beta}$ with $0 \le n \le M$, the vector $v(\beta,n,T)$ (defined in \eqref{vec:v-beta-n}) is contained in $v(\mathcal{C}, T)$. Let $\Delta$ be the smallest positive distance between the set of iterates $T^n({\beta})$, where $\beta \in \mathcal{C}$ and $0 \le n \le M$, and the set of discontinuities $\mathcal{C}$. Note that the smallest distance between those sets can be zero if there are critical connections. That is why in the definition of $\Delta$ we focus only on the positive distances.

Define $\epsilon_0 > 0$ to be sufficiently small so that the following holds. For any perturbation $\delta$ of the parameters $(\gamma \, \beta)$ of $T$ such that $|\Tilde{\gamma}_i - \gamma_i| < \epsilon_0$ for $1 \le i \le r$ and $|\Tilde{\beta}_i - \beta_i| < \epsilon_0$ for $1 \le i \le r-1$, where $(\Tilde{\gamma} \, \Tilde{\beta}) := (\gamma \, \beta) + \delta$, we have that for every vector $v \in v(\mathcal{C},T)$:
\[
\langle v , \delta \rangle < \Delta/4.
\]
If we denote the perturbed map by $\Tilde{T}$, this implies:
\begin{align*}
d(\Tilde{T}^n(\Tilde{\beta}), \Tilde{\beta}_*) &> d(T^n(\beta), \betas) - d(\betas, \Tilde{\beta}_*) - d(\Tilde{T}^n(\Tilde{\beta}), T^n(\beta)) \\
&> \Delta - \frac{\Delta}{4} - \frac{\Delta}{4} \\ &> \frac{\Delta}{2} > 0,
\end{align*}
for every $\beta, \betas \in \mathcal{C}$ and $0 \le n \le M$, since $\langle v(\beta,n,T) \,, \delta \rangle = \Tilde{T}^n(\Tilde{\beta}) - T^n(\beta)$.
Then by the definition of $M$ and $\Delta$, this means that the itinerary of every discontinuity $\beta$ of $T$ does not change at the times $n \le M$ for which $T^n{\beta} \notin \mathcal{C}$, for every such sufficiently small perturbation $\delta$. In other words, the itinerary can only change at times $n \le M$ for which $T^n{\beta} \in \mathcal{C}$, which can happen for arbitrarily small perturbations. By making $\epsilon_0$ smaller if necessary, we can extend the same conclusion to backward iterates of $\beta \in X$, i.e.\ iterates of $\beta$ under $T^{-1}_{\vert X}$ for times $n \le M$. This ensures that each such $\Tilde{\beta}$ still has a $\Tilde{T}^{-n}$-preimage for all times $n \le M$ such that $T^{-n}(\beta) \notin \mathcal{C}$ and that these preimages have the same itineraries up to the landing time on $\Tilde{\beta}$.

Let $0 < \epsilon < \epsilon_0$ and  $\epsilon^{\gamma}_1, \dots \epsilon^{\gamma}_{N}, \epsilon^{\beta}_0, \dots \epsilon^{\beta}_{N} \in (-\epsilon, \epsilon)$ be arbitrary real numbers. By Corollary \ref{cor:lin-indep}, the vectors:
\begin{align}
\label{eq:lin-indep-vectors}
\begin{split}
&L_{j}^0, R^{0,+}_j, C^{0,+}(j,k), C^{0,-}(j,k), \\
&C^{i,+}(0,k) \text{ for } 1 \le k \le m_0^{i,+}, \\
&C^{i,-}(1,k) \text{ for } 1 \le k < m_0^{i,-}, \\
&C_{\beta}, \text{ for } \beta \in \mathcal{C}_{\neg X}.
\end{split}    
\end{align}
form a linearly independent set. Thus there exists a vector $\delta \in \mathbb{R}^{2r-1}$ such that:
\begin{align}
\label{vec:l}
& \langle L_{j}^0, \delta \rangle = \epsilon_j^{\beta}, \\
\label{vec:r}
& \langle R_{j-1}^{0,+}, \delta \rangle = \epsilon_j^{\gamma} + \epsilon_{j-1}^{\beta}, \text{ for } 1 \le j \le N_0.  
\end{align}
Moreover, the following values can be chosen arbitrarily within the specified bounds:
\begin{align}
\label{vec:c0+}
&\langle C^{0,+}(j,k), \delta \rangle \in [0,\epsilon), \\
\label{vec:c0-}
&\langle C^{0,-}(j,k), \delta \rangle \in (-\epsilon,0], \\
\label{vec:ci}
&\langle C^{i,+}(0,k), \delta \rangle, \langle C^{i,-}(1,k), \delta \rangle \in (-\epsilon, \epsilon) \text{ for } i \ge 1, \\
\label{vec:cnX}
& \langle C_{\beta}, \delta \rangle \in (-\epsilon,\epsilon) \text{ for } \beta \in \mathcal{C}_{\neg X}.
\end{align}
The existence of such a vector $\delta$ is standard linear algebra, and it can be constructed using the Gram matrix of the vectors in \eqref{eq:lin-indep-vectors}. Thus $\delta$ can be chosen to linearly depend on $\epsilon, \epsilon^{\gamma}_1, \dots \epsilon^{\gamma}_{N}, \epsilon^{\beta}_0, \dots \epsilon^{\beta}_{N}$, so we clearly have that $|\Tilde{T} - T| = |\delta| \to 0$ as $\epsilon \to 0$.

We now show that this is the desired perturbation. By our choice of $\epsilon_0$, each discontinuity $\beta(j)$, for $1 \le j \le N_0-1$, still has a $\Tilde{T}^{l_j}$-preimage $\Tilde{a}_j^0$ that has the same itinerary up to time $l_j$ as $a_j$. Since $\langle L_j^{0}, (\gamma \, \beta) \rangle = a_j^0$, \eqref{vec:l} implies that $\Tilde{a}_j^0 - a_j^0 = \epsilon_j^{\beta}$. A similar argument shows that $\Tilde{a}_0^{+}-a_0^{+} = \epsilon_0^{\beta}$ and $\Tilde{a}_N^{-}-a_N^{-} = \epsilon_N^{\beta}$, so item 1. from the conclusion follows.

Assume now that $a_0^+$ does not land on a discontinuity before returning to $J$, with the case for $a_n^-$ being analogous. In this case, there is an interval of points $[a_0^+-\epsilon', a_0^+]$ that maps forward continuously under iterates of $T$ up to time $r_0$. By choosing $\epsilon_0 < \epsilon'$ small enough, this interval also maps forward continuously under iterates of $\Tilde{T}$ up to time $r_0$. Up to making $\epsilon'$ smaller, we may therefore choose $\Tilde{a}_0^+$ to be any point in $(a_0^+-\epsilon_0^{\beta}, a_0^+ + \epsilon_0^{\beta})$. In \eqref{eq:lin-indep-vectors}, we can simply omit the vectors $L^0_j$ and $C^{0,+}(0,k)$, and the rest of the analysis remains unchanged. Thus item 1. follows in this case as well. 

For 2., recall from \eqref{eq:R-vec-property} that $Tr(J_j,r_j) + a_{j-1}^{+} = T^{r_j}(a_{j-1}^+) = R_J(a_{j-1}^+) = \langle R_{j-1}^{+}, (\gamma \, \beta) \rangle$, so we have that:
\begin{align*}
\Tilde{T}r(\Tilde{J}_j,r_j) - Tr(J_j,r_j) &= \Tilde{T}^{r_j}(\Tilde{a}_{j-1}^+) - \Tilde{a}_{j-1}^+ - (T^{r_j}(a_{j-1}^+) - a_{j-1}^+) \\
&= \Tilde{T}^{r_j}(\Tilde{a}_{j-1}^+) - T^{r_j}(a_{j-1}^+) + a_{j-1}^+ - \Tilde{a}_{j-1}^+ \\
&= \langle R_{j-1}^{+}, \delta \rangle - \epsilon_{j-1}^{\beta} \\
&= \epsilon_{j}^{\gamma},
\end{align*}
where the last line follows from \eqref{vec:r}. Thus 2. follows. Finally (a) follows from \eqref{vec:c0+} and \eqref{vec:c0-}, while (b) follows from \eqref{vec:ci} and \eqref{vec:cnX}.

Assume now that $J$ is a maximal periodic interval of $X$, with minimal period $r$. In this case we set $J_1 := J$, and $J_2, \dots, J_n$ equal to maximal periodic intervals with pairwise disjoint orbits that contain all of the periodic discontinuities of $X$ not contained in the orbit of $J_1$. We then apply Corollary \ref{cor:lin-indep} only to the intervals $J_1, \dots, J_n$, omitting the interval $J_0$, which can clearly be done. We choose the vector $\delta$ so that:
\begin{align}
\label{vec:c1+}
&\langle C^{1,+}(0,k), \delta \rangle \in [0,\epsilon) \text{ for } 1 \le k < m_0^{i,+}, \\
\label{vec:c1m+}
&\langle C^{1,+}(0,m_0^{i,+}), \delta \rangle = \epsilon^{\gamma}_1 - \sum_{k=1}^{m_0^{i,+}-1} \langle C^{1,+}(0,k), \delta \rangle, \\
\label{vec:c1-}
&\langle C^{1,-}(1,k), \delta \rangle \in (-\epsilon,0] \text{ for } 1 \le k < m_1^{i,-}, \\
\label{vec:ci>1}
&\langle C^{i,+}(0,k), \delta \rangle, \langle C^{i,-}(1,k), \delta \rangle \in (-\epsilon, \epsilon) \text{ for } i \ge 2, \\
\label{vec:cnX-again}
& \langle C_{\beta}, \delta \rangle \in (-\epsilon,\epsilon) \text{ for } \beta \in \mathcal{C}_{\neg X}.
\end{align}
Similarly as in the previous case, by our choice of $\epsilon_0$ and by \eqref{vec:c1+}, \eqref{vec:c1m+}, and \eqref{vec:c1-}, the itineraries of $a_0^+$ and $a_1^-$ remain the same after perturbation, so 1. follows. Since by \eqref{eq:R-vec-property} we have that $Tr(J,r) = R_J(a_0^+) - a_0^+ = \langle \sum_{k=1}^{m_0^{i,+}} C^{1,+}(0,k), \delta \rangle$, 2. follows from \eqref{vec:c1+} and \eqref{vec:c1m+}. Finally, (a) follows from \eqref{vec:c1+}, while (b) follows from \eqref{vec:ci>1} and \eqref{vec:cnX-again}.
\end{proof}

\section{Proof of Theorem \ref{thm:lin-dep}}
\label{sec:proof-of-lin-dep}

\subsection{Overview of the proof}
\label{subsec:lin-dep-strategy}

In this subsection, we give an informal description of the strategy of the proof of Theorem \ref{thm:lin-dep}. The main idea is that a linear dependence of the vectors in the statement of Theorem \ref{thm:lin-dep} gives a set of linear equations satisfied by the coefficients $\alpha \in \mathcal{A}$ associated to these vectors. One can then use the dynamics of $T$ to produce a larger set of such linear equations, \textit{still with the same set of coefficients} $\mathcal{A}$. Repeating this analysis inductively produces a large number of equations with coefficients in $\mathcal{A}$, from which we derive the required relationships between these coefficients from Theorem \ref{thm:lin-dep}. The proof is divided into two steps, which we describe in more detail in the remainder of this subsection. The common features of both steps are certain \textit{refinement procedures} on subintervals of $I$.

\subsubsection{First refinement procedure}
\label{subsubsec:first-refinement}

We will demonstrate the proof for a simple case of $4$ vectors $v_1, v_2, v_3, v_4$. Assume that they are linearly dependent with coefficients $\alpha_1, \alpha_2, \alpha_3, \alpha_4$:

\begin{equation}
\label{eq:examp-lin-dep}
\alpha_1 v_1 + \alpha_2 v_2 + \alpha_3 v_3 + \alpha_4 v_4 = 0.
\end{equation}

Each of the vectors $v_i$ corresponds to the forward orbit of some point $x_i$ up to some finite time $k_i$: $\{ x_i, T(x_i), \dots, T^{k_i - 1}(x_i) \}$. We consider these orbits as differently coloured points (later called distinguished points in Subsection \ref{subsec:refining-part}) in the dynamical plane: $v_1$ corresponds to \textcolor{black}{blue}, $v_2$ corresponds to \textcolor{black}{red}, $v_3$ corresponds to \textcolor{black}{green}, and $v_4$ corresponds to \textcolor{black}{yellow}. An example for $T$ on $r=6$ intervals is shown in Figure \ref{fig:coloured points in the dynamical plane}. Coloured points have the following dynamical property: almost every coloured point is mapped onto and gets mapped to a point of the same colour under $T$. The exceptions are the first and last coloured points in the corresponding orbit, which do not necessarily have a preimage or image of the same colour, respectively.

\begin{figure}[h]
    \centering
    \includegraphics[width=\linewidth]{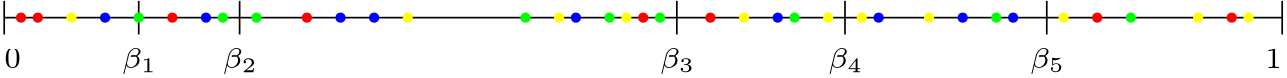}
    \caption{Coloured points associated to some map $T$ on $r=6$ intervals.}
    \label{fig:coloured points in the dynamical plane}
\end{figure}

For each of the intervals $I_1, \dots, I_6$ we may consider the number of points of a single colour contained in them: let $N_i(j)$ be the number of points corresponding to the $i$'th vector contained in $I_j$, where $1 \le i \le 4$ and $1 \le j \le 6$ (this notation will be slightly different in Subsection \ref{subsec:refining-part} because the setting will be more general). By definition, $N_i(j)$ is equal to the $j$'th $\gamma$-coefficient of the vector $v_i$. With this in mind, taking the $j$'th $\gamma$-coefficient in equation \eqref{eq:examp-lin-dep}, we get that:

\begin{equation}
\label{eq:examp-lin-eq}
\alpha_1 N_1(j) + \alpha_2 N_2(j) + \alpha_3 N_3(j) + \alpha_4 N_4(j) = 0.
\end{equation}

For example, in Figure \ref{fig:coloured points in the dynamical plane} and for $j=3$, this translates to:

\begin{equation}
\label{eq:examp-lin-eq-j=3}
\alpha_1 3 + \alpha_2 2 + \alpha_3 4 + \alpha_4 3 = 0.
\end{equation}

Thus the linear dependence in equation \eqref{eq:examp-lin-dep} gives us a set of equations as in \eqref{eq:examp-lin-eq} with the same coefficients $\alpha_1, \alpha_2, \alpha_3, \alpha_4$. Each equation is associated with one of the intervals $I_s$, which form a partition of $I$. In order to establish a stronger relationship between these coefficients, as is required for Theorem \ref{thm:lin-dep}, we will use the dynamics of $T$ to non-trivially refine this partition into a new partition which has one interval more. This new partition will also have the property that associated with each interval, there is an equation with coefficients $\alpha_1, \alpha_2, \alpha_3, \alpha_4$. Thus this refinement of the partition of $I$ into intervals $I_s$ produces one additional equation that holds for $\alpha_1, \alpha_2, \alpha_3, \alpha_4$. Our goal is to inductively produce more and more of such equations.

To do this, consider the open interval $G$ of maximal length between two coloured points that are contained in a single interval $I_s$. We may assume that the preimage of $G$ does not contain such a maximal interval (in general, $G$ needs to satisfy a number of assumptions, see Lemma \ref{lem:dist-pair}). In Figure \ref{fig:gap interval G}, the interval $G$ is contained in $I_3$. Let $I^l_3$ and $I^r_3$ be the half-open subintervals of $I_3$ to the left and to the right of the midpoint of $G$, respectively, as drawn in Figure \ref{fig:gap interval G}. We claim that these intervals also satisfy a version of equation \eqref{eq:examp-lin-eq} for the number of coloured points contained in them. Let $N^l_i(3)$ and $N^r_i(3)$ be the number of coloured points corresponding to $v_i$ contained in $I^l_3$ and $I^r_3$, respectively.

\begin{figure}[h]
    \centering
    \includegraphics[width=\linewidth]{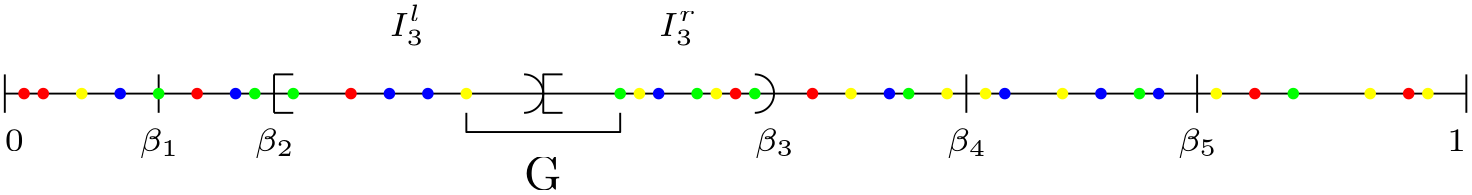}
    \caption{The gap interval $G$ and subintervals $I^l_3$ and $I^r_3$ for the map $T$ from Figure \ref{fig:coloured points in the dynamical plane}.}
    \label{fig:gap interval G}
\end{figure}

Because $T$ is an isometry, this interval $G$ has the following property: for all $1 \le s \le 6$, all coloured points inside $I_s$ \textit{all} get mapped (under a single iterate of $T$) to the same side of $G$, i.e.\ either to the left or to the right of $G$. Otherwise, if some coloured points get mapped to the left and some to the right, there would exist a gap interval of larger length. Because of this, we define the set $L$ of all indices $1 \le s \le 6$ for which all coloured points in $I_s$ get mapped to the left of $G$. Since almost all coloured points are mapped to and from the points of the same colour, the following almost-equality holds for $N_i^l(3)$:

\begin{equation}
\label{eq:examp-lin-eq-l-i}
N_i^l(3) \approx \left( \sum_{j \in L} N_i(j)\right) - N_i(1) - N_i(2),
\end{equation}
\noindent
since the number of coloured points to the left of the midpoint of $G$ should be equal to the number of coloured points that get mapped to the left of the midpoint of $G$. Equality does not hold in general, because, as noted earlier, the first and the last point in the orbit corresponding to a given colour do not necessarily have a preimage or image of the same colour, respectively. Because of this, the error in \eqref{eq:examp-lin-eq-l-i}, i.e.\ the difference between the two sides of the equation, is equal to either $-1$, $0$, or $1$. We define $\Delta_i^l \in \{-1,0,1\}$ as the error term of this equality, so that:

\begin{equation}
\label{eq:examp-delta}
N_i^l(3) = \left( \sum_{j \in L} N_i(j) \right) - N_i(1) - N_i(2) + \Delta_i^l.
\end{equation}
\noindent
Multiplying equation \eqref{eq:examp-delta} by $\alpha_i$, and summing over all $i$, we get that:

\begin{equation}
\label{eq:examp-lin-eq-l}
\sum_{i=1}^4 \alpha_i N_i^l(3) =\sum_{i=1}^4 \alpha_i \left( \left(\sum_{j \in L} N_i(j) \right) - N_i(1) - N_i(2) + \Delta_i^l\right).
\end{equation} 
\noindent
Rearranging the right-hand side of the equation, we have the following equality:

\begin{align}
\label{eq:examp-lin-eq-l-simplified}
\sum_{i=1}^4 \alpha_i N_i^l(3) &= \left( \sum_{j \in L} \sum_{i=1}^4 \alpha_i N_i(j) \right) - \sum_{i=1}^3 \alpha_i N_i(1) - \sum_{i=1}^4 \alpha_i N_i(2) + \sum_{i=1}^4 \alpha_i \Delta_i^l \\
&= \sum_{i=1}^4 \alpha_i \Delta_i^l,
\end{align}
\noindent
where the second line follows from \eqref{eq:examp-lin-eq}. Thus in order to show that $\sum_{i=1}^4 \alpha_i N_i^l(3) = 0$, we only need to show $\sum_{i=1}^4 \alpha_i \Delta_i^l = 0$. To prove this fact, we crucially use that there is only a single interval $J_0$ in Theorem \ref{thm:lin-dep} for which we consider the full orbit of the return map. Moreover, we make use of another set of the equations for $\alpha_1, \alpha_2, \alpha_3, \alpha_4$, obtained from \eqref{eq:examp-lin-dep} by considering all of the $\beta$-coefficients of $v_1, v_2, v_3, v_4$ (up to this point, all of our discussion involved only the $\gamma$-coefficients). The details are in the proof of Proposition \ref{prop:refine}.

The conclusion $\sum_{i=1}^4 \alpha_i N_i^r(3) = 0$, follows from $\sum_{i=1}^4 \alpha_i N_i^l(3) = 0$, if we subtract the latter equation from \eqref{eq:examp-lin-eq} for $j=3$. In Figure \ref{fig:coloured points in the dynamical plane}, these conclusions translate to:

\begin{align*}
\alpha_1 2 + \alpha_2 1 + \alpha_3 1 + \alpha_4 1 &= 0, \\
\alpha_1 1 + \alpha_2 1 + \alpha_3 3 + \alpha_4 2 &= 0,
\end{align*}
\noindent
or more formally:

\begin{align*}
\alpha_1 N_1^l(3) + \alpha_2 N_2^l(3) + \alpha_3 N_3^l(3) + \alpha_4 N_4^l(3) &= 0, \\
\alpha_1 N_1^r(3) + \alpha_2 N_2^r(3) + \alpha_3 N_3^r(3) + \alpha_4 N_4^r(3) &= 0.
\end{align*}

\begin{figure}[h]
    \centering
    \includegraphics[width=\linewidth]{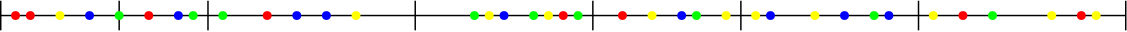}
    \caption{The refinement of the partition from Figure \ref{fig:coloured points in the dynamical plane}.}
    \label{fig:refined partition}
\end{figure}

We have thus increased the number of equations with coefficients $\alpha_1, \alpha_2, \alpha_3, \alpha_4$, as required. As a result, we have refined the partition of $I$ into continuity intervals $T$ into a partition that has one more interval, as in Figure \ref{fig:refined partition}. To proceed inductively, observe that in this analysis, it was at no point necessary to use the fact that the intervals $I_1, \dots, I_6$ are exactly the maximal intervals of continuity of $T$. It is sufficient that they are intervals on which $T$ is continuous and that the equations in \eqref{eq:examp-lin-eq} hold for them.

Thus we may proceed inductively for as long as we can find the required maximal intervals $G$. Ideally, the refinement procedure would continue until every subinterval of the partition would contain either one or two points, and at that point, Theorem \ref{thm:lin-dep} would follow. In the general case, the interval $G$ needs to satisfy several additional assumptions (see the statement of Lemma \ref{lem:dist-pair}), so this first refinement procedure stops at some maximal level, when such an interval $G$ can no longer be found. It can be shown that at this point, already every coloured point outside of $X$ is the only one contained in some interval of this partition of maximal level, so the conclusion of Theorem \ref{thm:lin-dep} for the vectors associated to discontinuities not contained in $X$, i.e.\ $\beta \in \mathcal{C}_{\neg X}$, already follows (Corollary \ref{cor:alpha_x=0}). Thus we may forget about these vectors and focus exclusively on the dynamics on $X$.

\subsubsection{Second refinement procedure}
\label{subsubsec:second-refinement}

The second part of the proof, contained in Subsection \ref{subsec:refine-alpha-subord}, starts with the observation (Corollary \ref{cor:a-subord-int-in-X}) that as a result of the first refinement procedure, each interval component $H$ of $X$ is partitioned into smaller subintervals $H_1, \dots, H_k$ for which the numbers of coloured points contained in them satisfy equations of type \eqref{eq:examp-lin-eq} with the same coefficients $\alpha_1, \alpha_2, \alpha_3, \alpha_4$. Such intervals are called $\mathcal{A}$-subordinate intervals (Definition \ref{def:alpha-subordinate}). The $T$-preimage of $H$ in $X$ is also partitioned into $\mathcal{A}$-subordinate intervals $G_1, \dots, G_l$. The case of $k=5$ and $l = 3$ is illustrated in Figure \ref{fig:alpha subordinate interval refinement}.

\begin{figure}[h]
    \centering
    \includegraphics[width=\linewidth]{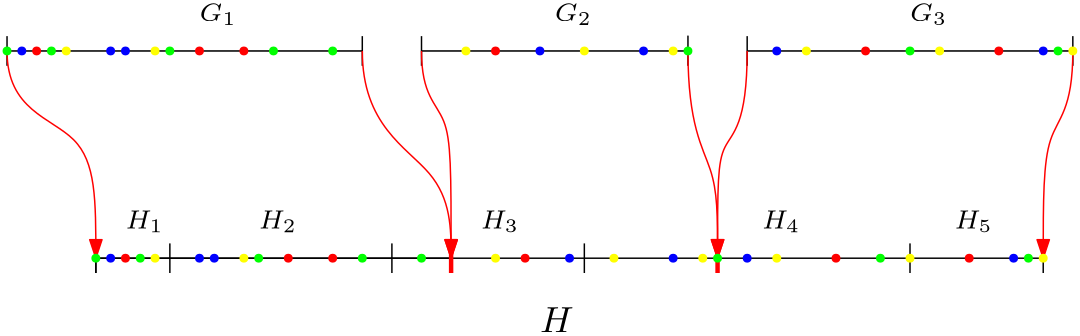}
    \caption{Example interval $H$ and its preimage in $X$ for $k=5$ and $l=3$.}
    \label{fig:alpha subordinate interval refinement}
\end{figure}

\noindent
One may then show (Proposition \ref{prop:push-fwd}) that the partition of the preimage of $H$ into $\mathcal{A}$-subordinate intervals can be `pushed forward' to obtain a refinement of the partition of $H$ into $H_1, \dots, H_k$. The refinement is obtained by adding the endpoints of the images of the intervals $G_1, \dots, G_l$, indicated by red vertical lines in Figure \ref{fig:alpha subordinate interval refinement}, to the partition into intervals $H_1, \dots, H_k$, whose endpoints are denoted by black vertical lines.

The proof is by induction: we start with the leftmost interval of the new partition of $H$ and move to the right. That the subsequent interval in the left-right order is $\mathcal{A}$-subordinate follows from the inductive hypothesis and calculations similar to those in \eqref{eq:examp-lin-eq-l-simplified}. Finally, in Subsection \ref{subsec:proof-of-lin-dep} we repeatedly apply the second refinement procedure to finalise the proof of Theorem \ref{thm:lin-dep}.

\subsection{Refining partitions of $I$}
\label{subsec:refining-part}

For the purpose of the proof, we will simplify the notation used in the statement of Theorem \ref{thm:lin-dep}. Denote by $\mathcal{V}$ the set of all vectors from \hyperref[eq:vectors]{(*)} and by $\mathcal{A}$ the set of all coefficients from \hyperref[eq:a-coefficients]{(**)}. 

A critical connection vector $v \in \mathcal{V}$ corresponds to the orbit of a discontinuity of $T$ up to the time of landing on another discontinuity. For such $v$, let $P(v)$ be the ordered set of points along the orbit corresponding to $v$. The first landing vectors correspond to the orbit of a discontinuity of $R_{J}$ up to the first time of landing onto discontinuity of $T$, critical connection vectors corresponding to the orbit of a discontinuity of $T$ up to the landing time to another discontinuity of $T$, and the return vectors correspond to the orbit of a discontinuity up to the time of landing into $J$. We adopt the convention that the first point in the orbit corresponding to the vector is included in $P(v)$, while the last one is not. For example, this means that the $P(v)$ associated to the critical connection vector $C^{0,+}(0,1)$ is equal to $\{ \beta^{0,+}(0,1), T(\beta^{0,+}(0,1)), \dots, T^{q^{0,+}(0,2) - 1}(\beta^{0,+}(0,1)) \}$. If a discontinuity of $T$ is already contained in $J_0$, then the set $P(v)$ for the associated first landing vector is empty. We will exclude such vectors from the set $\mathcal{V}$. For a vector $v \in \mathcal{V}$, let us denote by $v_{first}$ the first element of $P(v)$ and by $v_{last}$ the last element of $P(v)$.

\begin{definition}
Let $\mathcal{V}$ be the set of all vectors as defined above. Then we define the set of all \textit{distinguished points} (with respect to $\mathcal{V}$) as the union of all points in $P(v)$:
\[
\mathcal{P} := \bigcup_{v \in \mathcal{V}} P(v).
\]
\end{definition}
Note that $\mathcal{P}$ is by definition a set that includes both signed (from critical connection and return vectors) and geometric (from first landing vectors) points. For a vector $v \in \mathcal{V}$, we will denote by $\alpha(v)$ its corresponding coefficient. The assumption \eqref{eq:lin-dep-sum} of Theorem \ref{thm:lin-dep} then reads as:

\begin{equation}
\label{eq:lin-dep}
\sum_{v \in \mathcal{V}} \alpha(v) v = 0.
\end{equation}
Let $v_s$ be the $s$-th $\bm{e}$-coefficient of $v$. Equation \eqref{eq:lin-dep} implies:

\[
\sum_{v \in \mathcal{V}} \alpha(v) v_s = 0, \text{ for all } 1 \le s \le r.
\]
Note that $v_s$ also equals the number of points from $P(v)$ contained in $I_s$. This means that we have the initial partition $\mathcal{I}^0$ of $I$ into intervals $I_s^0 := I_s$, for $1 \le s \le r$, such that a linear equation on the number of points from each $P(v)$ holds on each interval of this partition. This motivates the following two definitions:

\begin{definition}
\label{N}
Let $I^*$ be any interval in $I$. Then we define the number of points from $P(v)$ in $I^*$ as:
\[
N(I^*,v) := |P(v) \cap I^*|.
\]
\end{definition}

Next, we introduce intervals for which the number of points from each $P(v)$ satisfies a linear equation with respect to the coefficients in $\mathcal{A}$:

\begin{definition}[$\mathcal{A}$-subordinate interval]
\label{def:alpha-subordinate}
Let $I^*$ be an interval in $I$. We say that $I^*$ is an \textit{$\mathcal{A}$-subordinate interval} if the following equation holds: 
\[
\sum_{v \in \mathcal{V}} \alpha(v) N(I^*,v) = 0.
\]
\end{definition}

Let $\mathcal{I}^* = \bigcup^k_{s=1} I_s^*$ be a partition (not necessarily dynamical) of $I$ into intervals $I_s^*$, i.e.\ $I = \bigsqcup_{s=1}^k I_s^*$. We say $\mathcal{I}$ is an \textit{$\mathcal{A}$-subordinate partition} if every interval $I_s^*$ is an $\mathcal{A}$-subordinate interval. By the discussion above, $\mathcal{I}^0$ is an $\mathcal{A}$-subordinate partition. Our goal will be to refine the level-$0$ partition $\mathcal{I}^0$ into $\mathcal{A}$-subordinate partitions which consist of more intervals.

More precisely, we will define a sequence of partitions $\mathcal{I}^m = \bigcup^{r+m}_{s=1} I^m_s$ of $I$ into intervals $I_s^m$ so that:

\begin{enumerate}
    \item $\mathcal{I}^m$ is a refinement of $\mathcal{I}^{m-1}$;
    \item $\mathcal{I}^m$ is an $\mathcal{A}$-subordinate partition.
\end{enumerate}
This procedure will terminate at some maximal level $M$, described in Definition \ref{def:max-part}. Our first goal is to find a criterion for when the partition $\mathcal{I}^m$ can be refined into $\mathcal{I}^{m+1}$ by splitting some interval $I_s^m$ of $\mathcal{I}^m$ into two intervals. We will show that this can be done by splitting along the gap interval between a certain pair of two neighbouring, but not touching, distinguished points. Recall that two signed points $x$ and $y$ touch if $\{ x,y \}$ = $\{z^+, z^-\}$ for some point $z$. The following lemma shows how to find such a pair:

\begin{lemma}[Distinguished pair]
\label{lem:dist-pair}
Let $\mathcal{I}^m$ be an $\mathcal{A}$-subordinate partition. Assume that there exists a triple $(I_s^m, p_1, p_2)$, where $I^m_s$ is an element of the partition, and $p_1,p_2$ is a pair of signed and distinct distinguished points in $I^m_s$, with $p_1$ to the left of $p_2$, such that they are:
\begin{enumerate}[label=(\alph*)]
    \item Neighbouring, but not touching, in $I_s^m$;
    \item There exist no $1 \le j \le N_0-1$ and $0 \le k < r_j^0$ such that $p_1$ and $p_2$ are contained in $T^k(J_j^0)$, and no $i \ge 1$ and $k \ge 0$ such that $p_1$ and $p_2$ are contained in $T^k(J_i)$.
\end{enumerate}
Then one can find a triple, still denoted by $(I^m_s, p_1, p_2)$, that additionally satisfies:

\begin{enumerate}[label=(\alph*), resume]
    \item The gap interval $G = [p_1,p_2]$ between the points has maximum length among all triples satisfying (a) and (b);
   
    \item Either at least one of the points $p_1$, $p_2$ has no preimage in $\mathcal{P}$ or no pair of preimages $q_1, q_2 \in \mathcal{P}$ of $p_1, p_2$ is contained in the same element of the partition $\mathcal{I}^m$.
\end{enumerate}
\end{lemma}
\noindent
We call a pair of points from the conclusion of Lemma \ref{lem:dist-pair} a \textit{distinguished pair}.

\begin{proof}

By assumption, we can always choose $p_1$ and $p_2$ such that they satisfy (a) and (b), and by choosing the pair with maximal distance, we can assume they also satisfy (c). Assume that this pair does not satisfy assumption (d). We will show that such a pair can be `pulled back' to find a different pair of points that also satisfies (d).

By assumption, the pair of points $p_1, p_2$ has at least one pair of $T$-preimages $q_1, q_2 \in \mathcal{P}$ that are contained in a single interval of $\mathcal{I}^m$. Let us show that this pair $q_1, q_2$ also satisfies properties (a), (b) and (c). Indeed, if they are contained in a single element of the partition, then $T$ maps $[q_1,q_2]$ continuously to $G$, as $\mathcal{I}^m$ is a refinement of $\mathcal{I}^0$. Thus $q_1$ and $q_2$ are not touching and not contained in an iterate of the form $T^k(J_j^0)$ or $T^k(J_i)$, so (b) follows. Moreover, if they are neighbouring, i.e.\ if there is no point from $\mathcal{P}$ between them, then they also satisfy (a) and (c). 

Assume the contrary, that there is some point $q \in \mathcal{P} \cap [q_1, q_2]$ between points $q_1$ and $q_2$. Since this interval maps forward continuously, $T(q) \in G$ is not contained in $\mathcal{P}$. This is only possible if $q = v_{last}$, where $v$ is either a critical connection vector corresponding to landing on a point in $\mathcal{C}_2$, or $q = v_{R,last}$, where $v$ is a return vector. In the first case, $T(q)$ must be a discontinuity, which contradicts $p_1$ and $p_2$ being in the same partition element of $\mathcal{I}^m$. In the second case, $T(q) \in J_0$, for some $1 \le i \le n$, so the interval $(p_1,p_2)$ intersects $J_0$. We will show that this means that $p_1$ and $p_2$ are not neighbouring in $\mathcal{P}$. Indeed, if neither of them is contained in $J_0 = [x^{0,+},y^{0,-}]$, then the boundary points of $J_0$ are between them. If only $p_1$ is contained in $J_0$, then $y^{0,-}$ is between them, since it is contained in $(p_1, p_2)$. If only $p_2$ is in $J_0$, then $x^{0,+}$ is between them. Finally, if they are both contained in $J_0$, then by (b) they are not contained in the same interval $J_j^0$, so they are not neighbouring in this case either. Thus the pair $q_1,q_2$ satisfies (c) as well. 

Thus, if the points $p_1, p_2$ satisfy (a), (b) and (c), but not (d), we can always find their pullback to another pair of points $q_1,q_2$ that also satisfies (a), (b) and (c). If the pair $q_1,q_2$ does not satisfy (d), then we can pull back this pair as well. If this can be performed indefinitely, then some pair of points $p_1$, $p_2$ appears twice. Then the interval $G = [p_1,p_2]$ is periodic, as the gap interval between the pullback of points maps continuously over the gap interval between the points themselves. Moreover, no point in $[p_1,p_2]$ lands on a discontinuity. Thus $p_1$ and $p_2$ are contained in $X$ and have the same itineraries. This is only possible if they are both contained in a single iterate of the form $T^k(J_j^0)$ or $T^k(J_i)$, as points from different iterates of this kind have different itineraries. This contradicts (b), so the pullback cannot be performed indefinitely. Thus, by pulling back, we arrive at a pair of points that also satisfies (d).
\end{proof}

Let $\mathcal{P}^m_s := I^m_s \cap \mathcal{P}$ and $G^{\circ} := G \setminus \{p_1, p_2\}$. Define $\frac{p_1+p_2}{2}$ as the average of geometric points in $I$ corresponding to $p_1$ and $p_2$, so that $\frac{p_1+p_2}{2}$ is the midpoint of $G$. Let $I^l$ be the interval in $I$ to the left of $\frac{p_1+p_2}{2}$ and let $I^r$ be the interval in $I$ to the right of $\frac{p_1+p_2}{2}$.

\begin{lemma}[Left or right]
\label{lem:l-or-r}
Let $\mathcal{I}^m$ be an $\mathcal{A}$-subordinate partition as in Lemma \ref{lem:dist-pair}, and let $G = [p_1, p_2]$ be a gap interval. Then for any element $I^m_s$ of this partition, the set $T(\mathcal{P}_s^m)$ is either fully to the left or fully to the right of $\frac{p_1+p_2}{2}$. More precisely, either $T(\mathcal{P}_s^m) \cap I^l = \emptyset$ or $T(\mathcal{P}_s^m) \cap I^r = \emptyset$.
\end{lemma}

\begin{proof}
Assume the contrary, that there exists an interval $I^m_s$ for which $T(\mathcal{P}_s^m)$ intersects both $I^l$ and $I^r$. Let $q_1, q_2 \in \mathcal{P}_s^m$ be two neighbouring (not necessarily signed) points in $I^m_s$ such that their images are contained in $\mathcal{P}$, and $q_1$ gets mapped to $I^l$ and $q_2$ gets mapped to $I^r$. Assume first that the image $T(q_1)$ is contained in $G^{\circ}$, with the case for $T(q_2)$ being analogous. Then $T(q_1)$ is not a point in $\mathcal{P}$, since $p_1$ and $p_2$ are neighbouring. This is only possible if $q_1 = v_{last}$, where $v$ is either a critical connection vector corresponding to landing on a point in $\mathcal{C}_2$ or a return vector. The second case was already shown to be impossible in the proof of Lemma \ref{lem:dist-pair}, while the first case is impossible since $p_1$ and $p_2$ are by assumption contained in a single interval of the partition $\mathcal{I}^m$. Thus, neither of the points $T(q_1), T(q_2)$ is contained in $G^{\circ}$. As this pair gets mapped forward continuously, this means that the image of the gap between these intervals is at least as large as $G$. By property (d) of Lemma \ref{lem:dist-pair}, it is impossible that $q_1$ and $q_2$ get mapped over $p_1$ and $p_2$, respectively. If $q_1$ and $q_2$ are both signed, then this is a contradiction with the maximality of $G$. We may, without loss of generality, assume that $q_1$ is not signed, with the other case being analogous. This means that $q_1 \in P(v)$ for a first landing vector $v$, so it is contained in the left boundary of an iterate of the form $T^k(J^0_j)$. As there is no other points in $\mathcal{P}$ contained in the interior $T^k(J^0_j)$, this means that $p_1$ and $p_2$ are both contained in $T^{k+1}(J^0_j)$, which is a contradiction.
\end{proof}

We now show that we can get an $\mathcal{A}$-subordinate partition of level $m+1$ if we split the interval $I^m_{s'}$ containing the pair $p_1,p_2$ along the gap interval $G$.

\begin{proposition}[Refining a partition]
\label{prop:refine} 
Let $\mathcal{I}^m$ be an $\mathcal{A}$-subordinate partition of level $m$ and assume that there exists at least one interval $I_{s'}^m$ of the partition as in Lemma \ref{lem:dist-pair}. Then the partition $\mathcal{I}^{m+1}$ obtained from $\mathcal{I}^m$ by splitting $I^m_{s'}$ along $\frac{p_1+p_2}{2}$ is an $\mathcal{A}$-subordinate partition of level $m+1$.
\end{proposition}

\begin{proof}
Let $I^{m,l}_{s'}$ be the maximal (half-open) interval to the left of $\frac{p_1+p_2}{2}$ inside $I^m_{s'}$, and let $I^{m,r}_{s'}$ be the maximal (half-open) interval to the right of $\frac{p_1+p_2}{2}$ inside $I^m_{s'}$. Then the intervals of the partition $\mathcal{I}^{m+1} = \bigcup_{1 \le s \le r+m+1} I^{m+1}_s$ are defined as:

\begin{equation*}
I^{m+1}_s := \begin{cases}
    I^m_s,& \text{For $s < s'$} \\[2pt]
    I^{m,l}_{s'},& \text{For $s = s'$} \\[2pt]
    I^{m,r}_{s'},& \text{For $s = s'+1$} \\[2pt]
    I^m_{s-1},& \text{For $s > s'+1$.}
\end{cases}
\end{equation*}
By construction, $\mathcal{I}^{m+1}$ is a refinement of $\mathcal{I}^0$. As only one of the intervals was modified, we know that:

\begin{equation}
\label{eq:m+1=m}
\begin{split}
\sum_{v \in \mathcal{V}} &\alpha(v) N(I^{m+1}_{s},v) = \sum_{v \in \mathcal{V}} \alpha(v) N(I^{m}_{s},v) = 0 \text{ for } s < s'; \\
\sum_{v \in \mathcal{V}} &\alpha(v) N(I^{m+1}_{s},v) = \sum_{v \in \mathcal{V}} \alpha(v) N(I^{m}_{s-1},v) = 0 \text{ for } s > s' + 1.
\end{split}
\end{equation}
Thus, in order to show that $\mathcal{I}^{m+1}$ is an $\mathcal{A}$-subordinate partition, we only need to prove:

\begin{equation}
\label{eq:l,r-eq}
\begin{split}
\sum_{v \in \mathcal{V}} &\alpha(v) N(I^{m+1}_{s'},v) = 0 \\
\sum_{v \in \mathcal{V}} &\alpha(v) N(I^{m+1}_{s'+1},v) = 0.
\end{split}
\end{equation}
It is enough to only prove one of the above equations holds, as the other follows from:

\begin{align*}
&\sum_{v \in \mathcal{V}} \alpha(v) N(I^{m+1}_{s'},v) + \sum_{v \in \mathcal{V}} \alpha(v) N(I^{m+1}_{s'+1},v) \\
= &\sum_{v \in \mathcal{V}} \alpha(v) N(I^{m}_{s'},v) = 0.
\end{align*}
Because of this, we may, without loss of generality, assume that the interval $J_0$ is contained in $I^r$, because we can otherwise just do the computation that follows for $I^r$ instead of $I^l$. This assumption will be important when we compute $\Delta_v$ (defined in \eqref{eq:n}) for  first landing and return vectors. Let $\mathcal{P}^l := \mathcal{P} \cap I^l$ be the set of all distinguished points to the left of the midpoint of $G$. Let Ind$^l$ be the set of all $s \in \{ 1, \dots, m+r \}$ such that $T(\mathcal{P}_s^m) \subset I^l$. This definition makes sense because of Lemma \ref{lem:l-or-r}. Let $\mathcal{P}^{l,-1} := \bigcup_{s \in \text{Ind}^l} \mathcal{P}^m_s = T^{-1}(\mathcal{P}^l) \cap \mathcal{P}$ be the set of all distinguished points that $T$ maps into $\mathcal{P}^l$. Consider the following equation:

\begin{equation}
\label{eq:n}
\Delta_v := \sum_{s \in \text{Ind}^l} N(I^{m}_s, v) - \sum_{s \le s'} N(I^{m+1}_s, v).
\end{equation}
This equation states that the number of elements of $P(v)$ contained in $I^l$ equals the number of elements of $P(v)$ that get mapped under $T$ into $I^l$ up to an error term $\Delta_v$. From \eqref{eq:n}, we get:

\begin{align*}
&\sum_{v \in \mathcal{V}} \alpha(v) N(I^{m+1}_{s'},v) \\
= &\sum_{v \in \mathcal{V}} \alpha(v) \left( \sum_{s \in \text{Ind}_L} N(I^m_s, v) - \sum_{s < s'} N(I^{m+1}_s, v) - \Delta_v \right) \\
= &\sum_{s \in \text{Ind}_L} \left( \sum_{v \in \mathcal{V}} \alpha(v) N(I^m_s, v) \right) \\
&- \sum_{s < s'} \left( \sum_{v \in \mathcal{V}} \alpha(v) N(I^{m+1}_s, v) \right) - \sum_{v \in \mathcal{V}} \alpha(v) \Delta_v \\
= &-\sum_{v \in \mathcal{V}} \alpha(v) \Delta_v.
\end{align*}
where the sum in the third line is zero because of the inductive assumption that each interval $I^m_s$ is $\mathcal{A}$-subordinate, and the sum in the fourth line is zero because of equation \eqref{eq:m+1=m}. Thus we only need to prove that:

\begin{equation}
\label{eq:delta-v}
\sum_{v \in \mathcal{V}} \alpha(v) \Delta_v = 0.
\end{equation}
We now compute the error term $\Delta_v$ for all vectors $v \in \mathcal{V}$. For this purpose, we define the following sets:

\begin{align*}
P(v)^{l,-1} &:= P(v) \bigcap \left( \bigcup_{s \in \text{Ind}^l} I^{m}_s \right)\\
P(v)^l &:= P(v) \bigcap \left( \bigcup_{s \le s'} I^{m}_s \right).
\end{align*}
By definition, the following holds:
\begin{align*}
|P(v)^{l,-1}| &= \sum_{s \in \text{Ind}^l} N(I^{m}_s, v) \\
|P(v)^l| &= \sum_{s \le s'} N(I^{m+1}_s, v).
\end{align*}
We first compute $\Delta_v$ for a critical connection vector $v \in \mathcal{V}$. Note that this includes the vectors in $\mathcal{C}_{\neg X}$. For any point $p$ in $P(v)^l$ that is not equal to $v_{first}$ (we do not assume it is contained in $P(v)^l$), there exists exactly one point $q$ in $P(v)^{l,-1}$ such that $T(q) = p$. Indeed, this is because each element of $P(v)$ except the first one has a preimage in $P(v)$. Moreover, $T(q) \in P(v)^l$ for each point $q$ in $P(v)^{l,-1}$ that is not equal to $v_{last}$. Thus we get the following formula for $\Delta_v$ for a critical connection vector $v \in \mathcal{V}$:

\begin{equation}
\label{eq:delta-v-crit}
\Delta_v = \begin{cases}
    0,& \text{If $v_{first} \in \mathcal{P}^l$ and $v_{last} \in \mathcal{P}^{l,-1}$} \\[2pt]
    -1,& \text{If $v_{first} \in \mathcal{P}^l$ and $v_{last} \notin \mathcal{P}^{l,-1}$} \\[2pt]
    1,& \text{If $v_{first} \notin \mathcal{P}^l$ and $v_{last} \in \mathcal{P}^{l,-1}$} \\[2pt]
    0,& \text{If $v_{first} \notin \mathcal{P}^l$ and $v_{last} \notin \mathcal{P}^{l,-1}$.}
\end{cases}
\end{equation}
Now we compute $\Delta_v$ for a return vector $v \in \mathcal{V}$. For every point $q$ in $P(v)^l$ that is not equal to $v_{first}$, we again have that there again exists exactly one point $q$ in $P(v)^{l,-1}$ such that $T(q) = p$. Moreover, $T(q) \in P(v)^l$ for each point $q$ in $P(v)^{l,-1}$ that is not equal to $v_{last}$. Recall that we have assumed that $J_0$ is contained in $I^r$. This means that $T(v_{last}) \notin \mathcal{P}^l$, and so $v_{last} \notin \mathcal{P}^{l,-1}$. Thus $\Delta_v$ in this case only depends on whether $v_{first} \in \mathcal{P}^l$, so we get the following formula:

\begin{equation}
\label{eq:delta-v-return}
\Delta_v = \begin{cases}
    -1,& \text{If } v_{first} \in \mathcal{P}^l \\[2pt]
    0,& \text{If } v_{first} \notin \mathcal{P}^l.
\end{cases}
\end{equation}
Analogously, for a first landing vector $v$, we know that $v_{first} \notin \mathcal{P}^l$, so the value of $\Delta_v$ depends only on whether $v_{last} \in \mathcal{P}^{l,-1}$:
\begin{equation}
\label{eq:delta-v-landing}
\Delta_v = \begin{cases}
    1,& \text{If $v_{last} \in \mathcal{P}^{l,-1}$} \\[2pt]
    0,& \text{If $v_{last} \notin \mathcal{P}^{l,-1}$}. \\[2pt]
\end{cases}
\end{equation}
Let $w_{\beta}$, for a (signed or geometric) discontinuity $\beta \in \mathcal{P}$, be the unique vector in $\mathcal{V}$ such that $w_{first} = \beta$, and let $w_{\beta} = 0, \alpha(w_{\beta}) = 0$ for all $\beta \notin \mathcal{P}$. We claim that the sum in \eqref{eq:delta-v} is equal to:

\begin{equation}
\label{eq:beta-L}
\sum_{\beta \in I^l} \left( - \alpha(w_{\beta^+}) - \alpha(w_{\beta^-}) + \sum_{\substack{w \in \mathcal{V}, \\ {T(w_{last}) \in \{ \beta, \beta^+, \beta^-\}}}} \alpha(w) \right).
\end{equation}
Here, we sum over all geometric discontinuities $\beta \in I^l$. Note for every $v \in \mathcal{V}$, the coefficient $\alpha(v)$ is equal to at most one $\alpha(w_{\beta^{\pm}})$ in \eqref{eq:beta-L}. Such a $\alpha(w_{\beta})$ exists if and only if $v_{first}$ is a signed discontinuity in $I^l$. Moreover, $\alpha(v)$ is also equal to at most one $\alpha(w)$ in the inner sum of \eqref{eq:beta-L}, and such $\alpha(w)$ exists if and only if $T(v_{last})$ is a (signed or geometric) discontinuity in $I^l$. We now prove that the following two claims hold:

\begin{enumerate}
    \item $\Delta_v = -1$ in \eqref{eq:delta-v} if and only if $\alpha(v)$ is equal to some $\alpha(w_{\beta})$ and not equal to any $\alpha(w)$ in \eqref{eq:beta-L};
    \item$\Delta_v = 1$ in \eqref{eq:delta-v} if and only if $\alpha(v)$ is not equal to any $\alpha(w_{\beta})$ and equal to some $\alpha(w)$ in \eqref{eq:beta-L}.
\end{enumerate} 
The equality of \eqref{eq:delta-v} and \eqref{eq:beta-L} clearly follows from these two claims, because then the contribution of $\alpha(v)$ to equations \eqref{eq:delta-v} and \eqref{eq:beta-L} is equal for all vectors $v \in \mathcal{V}$. We prove the first claim, while the proof of the second one is analogous.

If $\Delta_v = -1$, then $v_{first} \in \mathcal{P}^l$ and $v_{last} \notin \mathcal{P}^{l,-1}$. In particular, $v$ is not a first landing vector, by \eqref{eq:delta-v-return}. Thus $v_{first}$ is a signed discontinuity in $I^l$, while $T(v_{last})$ is not a discontinuity in $I^l$, so $\alpha(v)$ is equal to some $\alpha(w_{\beta})$ and not equal to any $\alpha(w)$. Conversely if $\alpha(v)$ is equal to some $\alpha(w_{\beta})$ and not equal any $\alpha(w)$ then the $v_{first}$ is a signed discontinuity in $I^l$, so $v_{first} \in \mathcal{P}^l$, and $T(v_{last})$ is not a discontinuity in $I^l$, so $v_{last} \notin \mathcal{P}^{l,-1}$. Thus $\Delta_v = -1$, so the claim follows.

We now show that the sum in \eqref{eq:beta-L} is equal to zero, which finishes the proof. Indeed, the $\bm{f}$-coefficient at ind$(\beta)$ for some $\beta \in I^l$ for the resulting vector in \eqref{eq:lin-dep} has to be zero. The only vectors in this sum for which this $\bm{f}$-coefficient is non-zero are those vectors $v$ for which $v_{first} = \beta^{\pm}$ or those for which $T(v_{last}) \in \{ \beta, \beta^+, \beta^- \}$. Thus taking the $\bm{f}$-coefficient at ind$(\beta)$ in equation \eqref{eq:lin-dep}, we get that:

\[
\left( - \alpha(w_{\beta^+}) - \alpha(w_{\beta^-}) + \sum_{\substack{w \in \mathcal{V}, \\ {T(w_{last}) \in \{ \beta, \beta^+, \beta^- \}}}} \alpha(w) \right) = 0
\]
As we split along an interval $G^{\circ}$ that does not contain a discontinuity in its interior, $\beta^+ \in I^l$ if and only if $\beta^- \in I^l$. Thus we have that:
\[
\sum_{\beta \in I^l} \left( - \alpha(w_{\beta^+}) - \alpha(w_{\beta^-}) + \sum_{\substack{w \in \mathcal{V}, \\ {T(w_{last}) \in \{ \beta, \beta^+, \beta^-\}}}} \alpha(w) \right) = 0,
\]
which finishes the proof.
\end{proof}

Let us now show that Proposition \ref{prop:refine} can be applied only finitely many times.

\begin{lemma}
\label{lem:finite-refine}
We can apply Proposition \ref{prop:refine} finitely many times if we start with the initial partition $\mathcal{I}^0$ given by the continuity intervals.
\end{lemma}

\begin{proof}
We will show that applying Proposition \ref{prop:refine} always increases the number of elements of the partition that have a non-empty intersection with $\mathcal{P}$. Since $\mathcal{P}$ is a finite set, this clearly shows that Proposition \ref{prop:refine} can only be applied finitely many times. By construction, we can apply Proposition \ref{prop:refine} to a partition $\mathcal{I}^m$ if and only if we find a pair of distinguished points $p_1,p_2$ as in Lemma \ref{lem:dist-pair}. These points are contained in the same element $I^m_{s'}$ of the partition, so after splitting the partition along the gap interval $G$ between them, we get two intervals $I^{m+1}_{s'}$ and $I^{m+1}_{s'}$ that contains $p_1$ and $p_2$, respectively. Thus we have increased the number of elements of the partition that have a non-empty intersection with $\mathcal{P}$, so the proof follows.
\end{proof}
 
Thus this refinement procedure stops at some maximal level:

\begin{definition}[Maximal partition]
\label{def:max-part}
Let $M$ be the level at which the procedure of inductively applying Proposition \ref{prop:refine} stops. We call the partition $\mathcal{I}^M$ obtained by this procedure the \textit{maximal partition}.
\end{definition}
It is of maximal level in the sense that for the partition $\mathcal{I}^M$ of this level, there is no pair of points satisfying (a) and (b) from Lemma \ref{lem:dist-pair}. Recall that $\mathcal{C}_{\neg X}$ is the set of all discontinuities in $\mathcal{C}_1$ not contained in $X$. The maximality of $\mathcal{I}^M$ gives the following corollary:

\begin{corollary}
\label{cor:alpha_x=0}
All of the coefficient $\alpha(v_{\beta})$, for $\beta \in \mathcal{C}_{\neg X}$, are equal to zero.
\end{corollary}

\begin{proof}
We will prove this for a $+$-type discontinuity $\beta$, with the proof for a $-$-type discontinuity being analogous. Let $\mathcal{I}^M$ be the maximal partition as in Definition \ref{def:max-part}. Assume that there is another point $p \in \mathcal{P}$ contained in the element of the partition $I^M_s$ containing $\beta$. As $\beta$ is a discontinuity, it must be the leftmost point $I^M_s$, as $\mathcal{I}^M$ is a refinement of $\mathcal{I}^0$. Thus we may assume that $\beta$ and $p$ are neighbouring, but not touching. As $\beta \notin X$, this pair of points is not contained in a periodic interval. Hence this pair of points satisfies properties (a) and (b) from Lemma \ref{lem:dist-pair}, so the partition $\mathcal{I}^M$ can in fact be refined, which is a contradiction with its maximality. Thus there is no other distinguished point in $I^M_s$. As $\mathcal{I}^M$ is an $\mathcal{A}$-subordinate partition, we have that $\alpha(v_{\beta}) = 0$, as this is exactly the equation which $I^M_s$ needs to satisfy to be $\mathcal{A}$-subordinate.
\end{proof}
As all of these coefficients are zero, they do not contribute to the equation \eqref{eq:lin-dep}. Thus we may exclude their corresponding vectors and points from $\mathcal{V}$ and $\mathcal{P}$, respectively. This means that all points in $\mathcal{P}$ are now contained in $X$. This immediately gives the following corollary, which proves a part of Theorem \ref{thm:lin-dep} for the first landing vectors corresponding to points in the interior of $J_0$:

\begin{corollary}
In conclusion \eqref{eq:lin-dep-equality2} of Theorem \ref{thm:lin-dep}, we have that $\alpha_j^0 = \alpha^{0,+}(j,1) + \alpha^{0,-}(j,1)$ for all $1 \le j \le N_0-1$.
\end{corollary}

\begin{proof}
By Corollary \ref{cor:alpha_x=0}, for all $1 \le j \le N_0-1$, there exist only three vectors $v \in \mathcal{V}$ such that the $\bm{f}$-coefficient at ind($\beta^0(j)$) is non-zero. These are exactly $L_j^0, C^{0,+}(j,1)$ and $C^{0,-}(j,1)$. As the $\bm{f}$-coefficients at ind($\beta^0(j)$) in equation \eqref{eq:lin-dep} have to cancel out, we get that:
\[
\alpha_j^0 = \alpha^{0,+}(j,1) + \alpha^{0,-}(j,1)
\]
for all $1 \le j \le N_0-1$.
\end{proof}
We may now assume that the points and vectors associated to the first landing vectors $L_j^0$, for all $1 \le j \le N_0-1$, are not contained in $\mathcal{P}$ and $\mathcal{V}$, respectively. Thus $\mathcal{P}$ now contains only signed points. We can make this assumption because we will be able to deduce the remaining equalities from the conclusions \eqref{eq:lin-dep-equality1} and \eqref{eq:lin-dep-equality2} of Theorem \ref{thm:lin-dep} without using these first landing vectors. 

The maximality of $\mathcal{I}^M$ also gives the following corollary:

\begin{corollary}
\label{cor:a-subord-int-in-X}
All maximal intervals contained in $X$ that do not contain discontinuities in their interiors are $\mathcal{A}$-subordinate intervals.
\end{corollary}
\begin{proof}
We claim that every element $I^M_s$ of the partition, $I^M_s \cap X$ contains at most one interval that contains points in $\mathcal{P}$. Indeed, if there were at least two such intervals in this intersection, then we could choose a pair of points from two such neighbouring intervals as the pair of points from Lemma \ref{lem:dist-pair}, which would mean we can further refine $\mathcal{I}^M$, and this is a contradiction with maximality of $\mathcal{I}^M_s$. Since $\mathcal{I}^M$ is a refinement of $\mathcal{I}^0$ and since in Proposition \ref{prop:refine} we always split along a gap interval $G$ not containing any points in $\mathcal{P}$ in its interior, a maximal interval $H$ contained in $X$ that does not contain discontinuities in its interior and has non-empty intersection with $\mathcal{P}$ is contained in a single element $I^M_s$ of the partition $\mathcal{I}^M$. By Corollary \ref{cor:alpha_x=0}, all of the points in $\mathcal{P} \cap I^M_s$ are contained in a single such interval $H$. Since $I^M_s$ is an $\mathcal{A}$-subordinate interval, this gives that $H$ is also an $\mathcal{A}$-subordinate interval. Since intervals contained in $X$ that do not contain points in $\mathcal{P}$ are trivially $\mathcal{A}$-subordinate, the proof follows.
\end{proof}

\subsection{Refining partitions into $\mathcal{A}$-subordinate intervals}
\label{subsec:refine-alpha-subord}

Our next goal is to partition the intervals of $X$ into smaller $\mathcal{A}$-subordinate intervals. These smaller intervals can then be further partitioned, and we will show that we may continue with this procedure until all intervals of the form $T^k(J_j^0)$, for $1 \le j \le N_0-1$ and $l_j^0 \le k < r_j^0$, and $T^k(J^i)$, for $1 \le i \le n$ and $k \ge 0$, are elements of this partition. To do this, we will show how a partition of an interval into smaller $\mathcal{A}$-subordinate intervals may be further refined by `pushing forward' the partition of the preimage of this interval.

Let $H = [a,b) \neq J_0$ be an interval component of $X$. Assume that $H$ satisfies the following two properties:
\begin{enumerate}
\label{H-properties}
\item There is a finite number of points $a = h_0 < h_1 < \dots < h_k < h_{k+1} = b$ such that the intervals $H_0 = [h_0, h_1), \dots, H_k = [h_k,h_{k+1})$ are $\mathcal{A}$-subordinate, and such the intervals $(h_q^+,h_{q+1}^-)$ contain no discontinuities for all $0 \le q \le k$; 
\item The preimage of $H$ in $X$ can be partitioned into $\mathcal{A}$-subordinate intervals $G_0 = [g_0,g_1), G_1=[g_2,g_3), \dots, G_l=[g_{2l},g_{2l+1})$, such that the intervals $(g_{2r}^+, g_{2r+1}^-)$ contain no discontinuities for all $0 \le r \le l$. We allow for $g_{2r+1}=g_{2r+2}$, i.e.\ the intervals could touch at their boundary points.     
\end{enumerate} 
We will refer to these intervals as $H$-intervals and $G$-intervals, respectively. We distinguish between the following two types of points in $H$:
\begin{itemize}
    \item $h_0^+,h_1^{\pm},\dots,h_{k+1}^-$ will be called \textit{Type-1 points};
    \item $T(g_0^+), T(g_1^-) \dots, T(g_{2l+1}^-)$ will be called \textit{Type-2 points}.
\end{itemize}
We will refer to any point that is either a Type-$1$ or a Type-$2$ point as \textit{special}. Note that a point can be both a Type-$1$ and a Type-$2$ point. For example, this is the case for boundary points $a^+$ and $b^-$.

Let $H_0' = [h_{0}',h_{1}'), H_1'=[h_{1}',h_{2}'), \dots, H_m'=[h_{m}',h_{m+1}')$ be the maximal intervals in $H$, such that the intervals $(h_{t}'^+,h_{t+1}'^-)$ contain no special points for all $0 \le t \le m$. In other words, the boundary of each $H_t'$ consists of two neighbouring, but not touching, special points. We will refer to these intervals as $H'$-intervals. These intervals form a partition of $H$, which is a refinement of the partition into $H_0, H_1, \dots, H_k$.

Recall that because of Corollary \ref{cor:alpha_x=0} all of the points in $\mathcal{P}$ are contained in $X$. Since $H \neq J_0$, every point $p \in \mathcal{P} \cap H$ not equal to some $\beta^{0,\pm}(j,1)$, for $1 \le j \le N_0-1$, has a unique preimage in $X \cap \mathcal{P}$. Moreover, $p$ is contained in the same $P(v)$ as its unique preimage in $X \cap \mathcal{P}$ if and only if $p$ is not a discontinuity of $T$. This fact will be repeatedly used in the proof below. It will also be convenient to introduce the following notation: for a $\beta \in X$ that is not equal to some $\beta^{0,\pm}(j,1)$, for $1 \le j \le N_0-1$, denote by $v^{-1}_{\beta}$ the unique vector in $\mathcal{V}$ such that $T(v_{last}) = \beta$. This vector is now unique because of Corollary \ref{cor:alpha_x=0}. For a discontinuity $\beta^{0,\pm}(j,1)$, with $1 \le j \le N_0-1$, we define $v^{-1}_{\beta^{0,\pm}(j,1)} = 0$. Recall also that $v_{\beta}$ is the unique vector in $\mathcal{V}$ such that $v_{first} = \beta$.

\begin{proposition}[Push-forward of a partition]
\label{prop:push-fwd}
Let $H$ satisfy the two assumptions in \ref{H-properties}. Then the intervals $H_0', \dots, H_m'$ are also $\mathcal{A}$-subordinate intervals. Moreover, for every discontinuity contained in $H$ that is also a Type-$2$ point, we have that:
\[
-\alpha(v^{-1}_{\beta}) + \alpha(v_{\beta}) = 0.
\]
\end{proposition}

\begin{proof}
The proof goes by induction, starting with the leftmost interval $H_0'$ and going to the right. The induction has two hypotheses:
\begin{itemize}
    \item $H_z'$ is an $\mathcal{A}$-subordinate interval for all $z \le t$;
    \item For each discontinuity $\beta$ of $T$ contained in these intervals that is also a Type-$2$ point, we have that:
    \[
    -\alpha(v^{-1}_{\beta}) + \alpha(v_{\beta}) = 0.
    \]
\end{itemize}
We now prove the base of the induction for $t=0$. By assumption, none of the $H$-intervals contains discontinuities in their interior. The proof depends on the type of the boundary point $h_1'$ of $H_0'$. There are $3$ different cases.

\underline{Case 1: $h_1'$ is only a Type-$1$ point}
\vspace{1mm}

Then $H_0' = H_0$, so the first hypothesis follows. Assume first that $a$ is not a discontinuity of $T$. In this case, even if $h_1'$ is a discontinuity of $T$, it is not of Type-$2$, so the second hypothesis follows. If $a^+ = \beta^+$, then since $H$ is assumed to be a component of $X$, we know that $\beta^- \notin X$. Thus the only vectors in $\mathcal{V}$ for which the $\bm{f}$-coefficient at ind($\beta^+$) is non-zero are $v_{\beta^+}$ and $v^{-1}_{\beta^+}$. Thus:
\[
-\alpha(v^{-1}_{\beta^+}) + \alpha(v_{\beta^+}) = 0,
\]
so the second hypothesis follows in this case as well.

\underline{Case 2: $h_1'$ is only a Type-$2$ point}
\vspace{1mm}

This means that $h_1'$ is not a discontinuity of $T$, as discontinuities are Type-$1$ points by the assumption that $H$-intervals do not contain discontinuities in their interiors. Moreover, $H_0' = T(G_r)$, for some $0 \le r \le l$. Assume that $a$ is not a discontinuity of $T$. Then the second hypothesis follows, as there are no discontinuities in $H_0'$. Moreover, all of the points in $H'_0 \cap \mathcal{P}$ belong to the same $P(v)$ as their preimages in $G_r$ and thus $N(H_0', v) = N(G_r, v)$ for all $v \in \mathcal{V}$. Thus $H'_0$ is an $\mathcal{A}$-subordinate interval, because $G_r$ is. Same as in the previous case, if $a^+ = \beta^+$, then:
\[
\alpha(v^{-1}_{\beta^+}) - \alpha(v_{\beta^+}) = 0,
\]
so the second hypothesis follows.

\underline{Case 3: $h_1'$ is both a Type-$1$ and a Type-$2$ point}
\vspace{1mm}

This means that $H_0' = H_0 = T(G_r)$, for some $0 \le r \le l$, so the first hypothesis follows. Assume that $a$ is not a discontinuity of $T$. If the right boundary point $h_1'$ of $H'_0$ is not a discontinuity of $T$, the second hypothesis trivially follows. If $h_1'^- = \betas^-$ is a discontinuity of $T$, then $N(H_0', v) = N(G_r, v)$ for all $v \in \mathcal{V}$, except for $v_{\betas}$ and $v^{-1}_{\betas}$. Thus:
\[
0 = \sum_{v \in \mathcal{V}} \alpha(v) N(G_r, v) - \sum_{v \in \mathcal{V}} \alpha(v) N(H'_0, v) = \alpha(v^{-1}_{\betas^-}) - \alpha(v_{\betas^-}),
\]
so the second hypothesis follows. If we assume $a^+ = \beta^+$, then as in Case $1$, we get that:
\[
\alpha(v^{-1}_{\beta^+}) - \alpha(v_{\beta^+}) = 0.
\]
If $h_1'$ is not a discontinuity of $T$, the second hypothesis follows, so we may assume $h_1'^-=\betas^-$. Then we have that:
\[
0 = \sum_{v \in \mathcal{V}} \alpha(v) N(G_r, v) - \sum_{v \in \mathcal{V}} \alpha(v) N(H'_0, v) = \alpha(v^{-1}_{\beta^+}) - \alpha(v_{\beta^+}) + \alpha(v^{-1}_{\betas^-}) - \alpha(v_{\betas^-}),
\]
which implies:
\[
\alpha(v^{-1}_{\betas^-}) - \alpha(v_{\betas^-}) = 0,
\]
so the second hypothesis follows.

Assuming the two induction hypotheses, we now want to prove that $H_{t+1}' = [h_{t+1}',h_{t+2}')$ is also an $\mathcal{A}$-subordinate interval, for some $t \ge 0$. The proof now depends on the type of both $h_{t+1}'^+$ and $h_{t+2}'^-$, so there are $4$ cases.

\underline{Case 1: Both are only Type-$1$ points}
\vspace{1mm}

Then $H_{t+1}' = H_q$ for some $q$, so we are done, because in this case, even if the boundary points are discontinuities, they are only of Type-$1$, so we need not check the second inductive hypothesis. 

\underline{Case 2: Both are only Type-$2$ points}
\vspace{1mm}

Then $H_{t+1}' = T(G_r)$ for some $r$, and neither of the boundary points is a discontinuity of $T$. Thus the first inductive hypothesis follows from the fact that $N(H_0', v) = N(G_r, v)$ for all $v \in \mathcal{V}$, and the second because there are no discontinuities in $H'_{t+1}$.

\vspace{1mm}

\underline{Case 3: $h_{t+1}'^+$ is a Type-$1$ point and $h_{t+2}'^-$ is a Type-$2$ point}

\vspace{1mm}

Then $H_{t+1}'$ is contained in the image of the interval $G_r$ for which $T(g_{2r+1}^-) = h_{t+2}'^-$. Let $H'_{u}, \dots, H'_{t}$ be all of the other $H'$-intervals contained in $T(G_r)$ and note that they form a partition of $T(G_r)$. Thus we have that:

\[
\sum_{z=u}^{t+1} N(H'_z, v) = N(G_r,v)
\]
for all $v$ except those for which either $v_{first} = \beta$ or $T(v_{last}) = \beta$, where $\beta$ is a discontinuity contained in $T(G_r)$. Hence the following holds:

\begin{align}
\label{eq:h-t+1}
\begin{split}
\sum_{v \in \mathcal{V}} \alpha(v) N(H'_{t+1}, v) &= \sum_{v \in \mathcal{V}} \alpha(v) N(G_r, v) - \sum_{z=u}^t \left( \sum_{v \in \mathcal{V}} \alpha(v) N(H'_z, v) \right) \\ &+ \sum_{\beta \in T(G_r)} \left( -\alpha(v_{\beta}) + \alpha(v^{-1}_{\beta}) \right) \\
&= \sum_{\beta \in T(G_r)} \left( -\alpha(v_{\beta}) + \alpha(v^{-1}_{\beta}) \right).
\end{split}
\end{align}
Here we sum over all signed discontinuities $\beta \in T(G_r)$. We now need to consider all of the discontinuities contained in $T(G_r)$. The analysis is different for discontinuities on the boundary and those in the interior. For a discontinuity $\beta$ contained in the interior of $T(G_r)$, we know that:

\begin{equation}
\label{eq:H'-int-disc}
-\alpha(v_{\beta^+}) + \alpha(v^{-1}_{\beta^+}) -\alpha(v_{\beta^-}) + \alpha(v^{-1}_{\beta^-}) = 0,
\end{equation}
by the assumption of Theorem \ref{thm:lin-dep} and by Corollary \ref{cor:alpha_x=0}. If the left boundary point $h'^+_{t+1}$ of $T(G_r)$ is a discontinuity $\betas^+$, then:
\begin{equation}
\label{eq:H'-bdry-disc}
-\alpha(v_{\betas^-}) + \alpha(v^{-1}_{\betas^-}) = 0,
\end{equation}
by the second inductive hypothesis. If the right endpoint $h'^-_{t+2}$ of $T(G_r)$ is a discontinuity $\betass^-$, then it is a Type-$1$ point as well, so $H'_{t+1}$ is an $H$-interval by the first assumption about the $H$-partition, and thus an $\mathcal{A}$-subordinate interval. Putting equations \eqref{eq:h-t+1}, \eqref{eq:H'-int-disc} and \eqref{eq:H'-bdry-disc} together gives:

\begin{align*}
0 = &\sum_{v \in \mathcal{V}} \alpha(v) N(H'_{t+1}, v) \\
\overset{\eqref{eq:h-t+1}}{=} &\sum_{\beta \in T(G_r)} \left( -\alpha(v_{\beta}) + \alpha(v^{-1}_{\beta}) \right) \\
\overset{\eqref{eq:H'-int-disc}\&\eqref{eq:H'-bdry-disc}}{=} &-\alpha(v_{\betas^+}) + \alpha(v^{-1}_{\betas^+}), 
\end{align*}
This proves the second inductive hypothesis, so the inductive step is done in this case. Thus we may assume that $h_{t+2}'$ is not a discontinuity. Then by equations \eqref{eq:h-t+1}, \eqref{eq:H'-int-disc} and \eqref{eq:H'-bdry-disc} we have that:

\begin{align*}
0 \overset{\eqref{eq:H'-int-disc} \& \eqref{eq:H'-bdry-disc}}{=} &\sum_{\beta \in T(G_r)} \left( -\alpha(v_{\beta}) + \alpha(v^{-1}_{\beta}) \right) \\
\overset{\eqref{eq:h-t+1}}{=} &\sum_{v \in \mathcal{V}} \alpha(v) N(H'_{t+1}, v), 
\end{align*}
so $H'_{t+1}$ is an $\mathcal{A}$-subordinate interval in this case as well. Finally, if the left boundary point $h'^+_{t+1}$ of $T(G_r)$ is not a discontinuity, then by equations \eqref{eq:h-t+1} and \eqref{eq:H'-int-disc} we have that:
\begin{align*}
0 \overset{\eqref{eq:H'-int-disc}}{=} &\sum_{\beta \in T(G_r)} \left( -\alpha(v_{\beta}) + \alpha(v^{-1}_{\beta}) \right) \\
\overset{\eqref{eq:h-t+1}}{=} &\sum_{v \in \mathcal{V}} \alpha(v) N(H'_{t+1}, v),
\end{align*}
which finishes the proof.

\underline{Case 4: $h_{t+1}'$ is a Type-$2$ point and $h_{t+2}'$ is a Type-$1$ point}
\vspace{1mm}

Then similarly as in the previous case, $H'_{t+1}$ is contained in a larger $H$-interval $H_q$, which is partitioned into $H'$-intervals $H'_u, \dots, H'_t$, which are $\mathcal{A}$-subordinate by the inductive hypothesis, and $H'_{t+1}$. Similarly to the previous case, we have that:

\begin{align*}
\sum_{v \in \mathcal{V}} \alpha(v) N(H'_{t+1}, v) &= \sum_{v \in \mathcal{V}} \alpha(v) N(H_q, v) - \sum_{z=u}^t \left( \sum_{v \in \mathcal{V}} \alpha(v) N(H'_z, v) \right) \\
&+ \sum_{\beta \in H_q} \left( -\alpha(v_{\beta}) + \alpha(v^{-1}_{\beta}) \right) \\
&= \sum_{\beta \in H_q} \left( -\alpha(v_{\beta}) + \alpha(v^{-1}_{\beta}) \right).
\end{align*}
From here we simply need to consider whether the boundary points of $H_q$ are discontinuities, and the proof follows analogously as in the previous case.
\end{proof}

It is interesting that the analogous statement for the `pull back' of a partition does not hold. A simple counterexample occurs when the interval $J_0$ does not contain a discontinuity, but its image $T(J_0)$ contains a single discontinuity $\beta$ in its interior. In this case $T(J_0)$ is partitioned into two $\mathcal{A}$-subordinate intervals $T(J_1)$ and $T(J_2)$, where $T(J_1)$ and $T(J_2)$ touch at $\beta$. Thus if we consider $H = T(J_0)$, the pull-back version of Proposition \ref{prop:push-fwd} would give us that $J_1$ and $J_2$ are $\mathcal{A}$-subordinate. This would mean that $\alpha_j^0 = \alpha^{0,+}(j,1)$ and $\alpha^{0,-}(j,1) = 0$, where $a_j$ is the discontinuity in $J-0$ that lands on $\beta$. By Theorem \ref{thm:lin-dep}, this does not necessarily hold, so in general we can not `pull back' a partition.

\subsection{Proof of Theorem \ref{thm:lin-dep}}
\label{subsec:proof-of-lin-dep}
In this subsection, we finish the proof of Theorem \ref{thm:lin-dep}. We do this by continuing with the process of refining partitions started in Subsection \ref{subsec:refining-part}. This will be done by refining the maximal partition $\mathcal{I}^M$ from Definition \ref{def:max-part}. This refinement is obtained by further partitioning the maximal intervals in $X$ that do not contain discontinuities in their interiors into smaller intervals.

As we will now care only about the partitions of the interval components of $X$, we introduce a new label for such partitions. We will say that a partition $\mathcal{I}_*$ of the components of $X$ into $\mathcal{A}$-subordinate intervals is an \textit{$\mathcal{A}X$-partition}. Such a partition is defined by a finite set of points that are contained in the interiors of intervals of $X$, which we call \textit{the break points} of $\mathcal{I}_*$. By Corollary \ref{cor:a-subord-int-in-X}, the partition of $X$ into maximal intervals that do not contain discontinuities in their interiors is an $\mathcal{A}X$-partition. As we do not need to consider $\mathcal{I}^M$ as a partition of $I$ anymore, we will denote this $\mathcal{A}X$-partition by $\mathcal{I}^M$ as well. Thus the set of break points $\mathcal{B}_0$ of $\mathcal{I}^M$ is the union of the boundary points of all maximal intervals contained in $X$ that do not contain a discontinuity in their interiors. This set is clearly the union of all discontinuities of $T$ and all the boundary points of $X$. We will show how to inductively refine this partition by applying Proposition \ref{prop:push-fwd}.

More precisely, we will show by induction that for each $m \ge 0$, there is an $\mathcal{A}X$-partition $\mathcal{I}^{M+m}$ such that its set of break points $\mathcal{B}_m$ is equal to the union of the following two types of points:

\begin{enumerate}
    \item Points $z \in X$ such that there exists a $q \le m$ and a discontinuity $\beta \in X$ such that $z = T^q(\beta)$ and $T^r(\beta) \notin J_0$ for all $0 < r \le q$;
    \item Points $z \in X$ such that there exists a $q \le m$ and a boundary point $x$ of $X$ such that $z = T^q(x)$, and $T^r(x) \notin J_0$ for all $0 < r \le q$.
\end{enumerate}
These two properties clearly imply that $\mathcal{B}_{m+1} \supset \mathcal{B}_m$, so $\mathcal{I}^{M+m+1}$ is a refinement of $\mathcal{I}^{M+m}$.

\begin{proof}[Proof of Theorem \ref{thm:lin-dep}]
The base of induction is given by Lemma \ref{lem:dist-pair} and Proposition \ref{prop:refine}. Assuming we have a level $m$ partition $\mathcal{I}^{M+m}$ such that its set of break points $\mathcal{B}_m$ is equal to the union of the two types of points described above, we show how to refine it into a level $m+1$ partition $\mathcal{I}^{M+m+1}$. We describe this refinement in terms of the new break points that get added to $\mathcal{B}_m$ by applying Proposition \ref{prop:push-fwd}.

Let $H$ be any component of $X$ not equal to $J_0$. By the inductive assumption, $\mathcal{I}^{M+m}$ induces a partition of $H$ into $\mathcal{A}$-subordinate intervals. We may take these $\mathcal{A}$-subordinate intervals to be the $H$-intervals from Proposition \ref{prop:push-fwd}. As $H$ is an interval component of $X$, its $T$-preimage in $X$ consists of maximal intervals that do not contain discontinuities in their interiors, which are $\mathcal{A}$-subordinate by Corollary \ref{cor:a-subord-int-in-X}. As $\mathcal{I}^{M+m}$ is a refinement of $\mathcal{I}^M$, these preimage intervals are further partitioned into $\mathcal{A}$-subordinate intervals, and we may take these intervals to be the $G$-intervals from Proposition \ref{prop:push-fwd}. Hence, by Proposition \ref{prop:push-fwd}, the partition into $H$-intervals can be refined into a partition into $\mathcal{A}$-subordinate $H'$-intervals. The set of boundary points of these $H'$-intervals is equal to the union of the images of break points in $\mathcal{B}_m$ contained in the $G$-intervals and the break points in $\mathcal{B}_m \cap H$. This holds for any $H \neq J_0$, and thus for each break point $p \in \mathcal{B}_m$, we get that $T(p)$ is a break point of $\mathcal{B}_{m+1}$, provided $T(p) \notin J_0$. Thus every point of the first and second type described above is contained in $\mathcal{B}_{m+1}$. Thus, the partition $\mathcal{I}^{M+m+1}$ obtained by applying Proposition \ref{prop:push-fwd} has the desired properties, which completes the induction step.

Thus, for every sufficiently large $m$, every point of either of these two types is a break point of $\mathcal{B}_{m}$:

\begin{enumerate}
    \item Points $z \in X$ such that there exists a $q$ and a discontinuity $\beta \in X$ such that $z = T^q(\beta)$ and $T^r(\beta) \notin J_0$ for all $0 < r \le q$;
    \item Points $z \in X$ such that there exists a $q$ and a boundary point $x$ of $X$ such that $z = T^q(x)$ and $T^r(x) \notin J_0$ for all $0 < r \le q$;
\end{enumerate}
Let $\beta^{0,+}(j,k)$, for $0 \le j \le N_0-1$ and $k < m^{0,+}_j$, be a discontinuity in the orbit of $J_0$, and let $P(v)$ be the orbit associate to the vector $v$ such that $v_{first} = \beta^{0,+}(j,k)$. Because of $k < m^{0,+}_j$, we know that $T(v_{last}) \in \mathcal{P}$. For $m$ large enough, we therefore know that $v_{last}$ is a break point of $\mathcal{B}_m$. Thus if we apply Proposition \ref{prop:push-fwd} to the interval component $H$ of $X$ containing $T(v_{last})$, we get that:
\begin{equation}
\label{eq:alpha-eqcc}
\alpha^{0,+}(j,k) = \alpha^{0,+}(j,k+1),    
\end{equation}
if $k+1 < m^{0,+}_j$, or
\begin{equation}
\label{eq:alpha-eqcr}
\alpha^{0,+}(j,m^{0,+}_j-1) = \alpha^{0,+}(j,m^{0,+}_j),
\end{equation}
if $k+1 = m^{0,+}_j$. This is because $T(v_{last}) = \beta^{0,+}(j,k+1)$ is a Type-$2$ point. An analogous argument applies to discontinuities of the form $\beta^{0,-}(j,k)$, for $1 \le j \le N_0$ and $k < m^{0,+}_j$, $\beta^{i,+}(0,k)$, for $1 \le i \le n$ and $k \le m^{i,+}_{0}$, and $\beta^{i,-}(1,k)$, for $1 \le i \le n$ and $k \le m^{i,-}_{1}$. Thus combining equations \eqref{eq:alpha-eqcc} and \eqref{eq:alpha-eqcr} for all of these discontinuities gives:
\begin{align}
\label{eq:alpha-eq-ccr}
\begin{split}
&\alpha^{i,+}(0,1) = \dots = \alpha^{i,+}(0,m_0^{i,+}) \\
&\alpha^{i,-}(1,1) = \dots = \alpha^{i,-}(1,m_{1}^{i,-}); \\
&\alpha^{0,+}(j,1) = \dots = \alpha^{0,+}(j,m_j^{0,+}-1) = a_j^{0,+}; \\
&\alpha^{0,-}(j+1,1) = \dots = \alpha^{0,-}(j+1,m_{j+1}^{0,-}-1) = \alpha_{j+1}^{0,-}; \\
&\alpha_0^0 = \alpha^{0,+}(0,1); \\
&\alpha_{N_0}^0 = \alpha^{0,-}(N_0,1), \\
\end{split}
\end{align}
for all $1 \le i \le n$ and $0 \le j \le N_0-1$, respectively.

Next, for every $0 \le j \le N_0$, both $a^{0,+}_j$ and $a^{0,-}_{j+1}$ have to land on discontinuities before the interval $[a^{0,+}_j,a^{0,-}_{j+1})$ returns to $J_0$. Thus there is a time $q$, such that the points $T^q(a^{0,+}_j)$ and $T^q(a^{0,-}_{j+1})$ are both break points of $\mathcal{B}_m$, for a sufficiently large $m$. Thus the interval $[T^q(a^{0,+}_j),T^q(a^{0,-}_{j+1}))$ is $\mathcal{A}$-subordinate, so we get that
\begin{equation}
\label{eq:alpha-eq+-}
\alpha(v) = -\alpha(w),
\end{equation}
where $v$ and $w$ are the vectors such that $T^q(a^{0,+}_j) \in P(v)$ and $T^q(a^{0,-}_{j+1}) \in P(w)$, respectively. Moreover, $v$ is either a critical connection vector or a return vector related to the orbit of $a^{0,+}_j$, and $w$ is either a critical connection vector or a return vector related to the orbit of $a^{0,-}_{j+1}$. Again, an analogous argument applies to discontinuities of the form $\beta^{0,-}(j,k)$, for $1 \le j \le N_0$ and $k < m^{0,+}_j$, $\beta^{i,+}(0,k)$, for $1 \le i \le n$ and $k \le m^{i,+}_{0}$, and $\beta^{i,-}(1,k)$, for $1 \le i \le n$ and $k \le m^{i,-}_{1}$. Thus, by combining \eqref{eq:alpha-eq-ccr} and \eqref{eq:alpha-eq+-} for all of these discontinuities, we get both of the conclusions \eqref{eq:lin-dep-equality1} and \eqref{eq:lin-dep-equality2} from the statement of Theorem \ref{thm:lin-dep}.
\end{proof}

\addcontentsline{toc}{section}{References}
\printbibliography

\end{document}